\documentclass[a4paper]{article}
\usepackage{etex}
\usepackage[utf8]{inputenc}
\usepackage[english]{babel} 
\usepackage{latexsym}
\usepackage{amsmath,amssymb,amsthm,amsfonts} 
\usepackage[overload]{empheq}
\usepackage{graphicx}
\usepackage[dvipsnames]{xcolor}
\usepackage{caption}
\usepackage{enumerate}
\usepackage[total={6in, 9.5in}]{geometry}
\usepackage{tikz}
\usepackage{bm}
\usepackage{enumitem}
\usepackage{algorithm}
\usepackage{algpseudocode}

\usepackage{pgfplots}
\usepackage{float}
\usetikzlibrary{intersections,shapes.arrows}
\usetikzlibrary{decorations.pathreplacing,decorations.markings}
\usetikzlibrary{arrows,decorations.markings}
\usetikzlibrary{spy}
\pgfplotsset{compat=1.18}
\usepgfplotslibrary{groupplots}
\usepgfplotslibrary{colorbrewer}
\pgfplotsset{
    cycle list/Set1,
}
\usepackage[font=footnotesize]{caption}

\newcommand{\ic}{{\rm i}}
\DeclareMathOperator{\rank}{rank}
\DeclareMathOperator{\Tr}{Tr}
\DeclareMathOperator{\DIV}{div}
\newcommand{\sDIV}{\nabla^{\perp}\cdot}
\newcommand{\grad}{\nabla}
\newcommand{\sgrad}{\grad^{\perp}}
\DeclareMathOperator{\curl}{curl}
\DeclareMathOperator{\Diff}{Diff}
\DeclareMathOperator{\St}{St}
\DeclareMathOperator{\iter}{it}
\DeclareMathOperator{\rf}{ref}

\DeclareMathOperator{\vc}{vec}
\DeclareMathOperator{\cay}{cay}
\DeclareMathOperator{\Iso}{\texttt{Iso-2}}
\DeclareMathOperator{\ad}{ad}
\DeclareMathOperator{\Ad}{Ad}
\DeclareMathOperator{\SU}{SU}
\DeclareMathOperator{\U}{U}
\newcommand{\bra}[2]{\ensuremath{\left\{#1,#2\right\}}}
\newcommand{\Stwo}{\mathbb{S}^2}
\newcommand{\fru}{\mathfrak{u}}
\newcommand{\frg}{\mathfrak{g}}
\newcommand{\Rcal}{\ensuremath{\mathcal{R}}}
\newcommand{\Ical}{\ensuremath{\mathcal{I}}}
\newcommand{\Tcal}{\ensuremath{\mathcal{T}}}
\newcommand{\Lcal}{\ensuremath{\mathcal{L}}}
\newcommand{\Fcal}{\ensuremath{\mathcal{F}}}

\newcommand{\norm}[1]{\lVert #1 \rVert}
\newcommand{\normF}[1]{\lVert #1 \rVert_F} 

\newcommand{\inprodL}[2]{\left<#1,#2\right>_{L^2(\Stwo)}}
\newcommand{\inprodF}[2]{\left<#1,#2\right>_F} 
\newcommand{\inprodN}[2]{\left<#1,#2\right>_{F,N}} 

\newcommand{\ubf}{\ensuremath{\bm{u}}}

\newcommand{\ebf}{\ensuremath{\bm{e}}}
\newcommand{\tind}{\ensuremath{\tau}}
\newcommand{\Ham}{H}
\newcommand{\Mr}{\mathcal{M}_r}
\newcommand{\ap}[1]{\widehat{#1}}
\newcommand{\apT}{\Tcal_{\ap{N}}}
\newcommand{\svd}[1]{\Pi_r(#1)}
\newcommand{\Wsvd}{\svd{W}}
\newcommand{\Wref}{W_{\rf}}
\newcommand{\coo}{\ensuremath{\Psi}}
\newcommand{\dcoo}{\ensuremath{\mathrm{d}\coo}}
\newcommand{\dt}{\ensuremath{\Delta t}}

\newcommand{\exd}{\mathsf{d}}
\newcommand{\contr}{\mathsf{i}} 

\usepackage[
	style=numeric,
	citestyle=numeric,sorting=nyvt,sortcites=false,
	giveninits=true,
	maxnames=10,
	url=false,eprint=true,doi=false,isbn=false,
	backref=false,
    backend=biber,
	natbib=true,
]{biblatex}
\renewbibmacro{in:}{} 
\AtEveryBibitem{\clearfield{note}} 
\AtEveryBibitem{\ifentrytype{book}{\clearfield{pages}}{}}
\usepackage[
	pdftex,
	bookmarks,
	colorlinks=true, 
	pdftitle={Low-rank Zeitlin model},
	pdfauthor={Cecilia Pagliantini},
	pdfsubject={},
	pdfkeywords={},
	urlcolor= blue,
	linkcolor= blue,
	citecolor= blue
]{hyperref}
\usepackage{cleveref}

\numberwithin{equation}{section}

\theoremstyle{plain}

\newtheorem{lemma}{Lemma}[section]
\newtheorem{proposition}{Proposition}[section]
\newtheorem{corollary}{Corollary}[section]
\newtheorem{remark}{Remark}[section]

\newcommand{\email}[1]{\protect\href{mailto:#1}{#1}}

\title{Geometric low-rank approximation of the Zeitlin model of incompressible fluids on the sphere}

\author{Cecilia Pagliantini\thanks{Dipartimento di Matematica,
			  Universit\`a di Pisa, Pisa,
			  Italy.
  (\email{cecilia.pagliantini@unipi.it}).\\
  Funding from the MIUR Excellence Department Project awarded to the Department of Mathematics, University of Pisa, CUP I57G22000700001, and from the INDAM/GNCS 2024 project CUP E53C23001670001 are acknowledged.}
}
\date{}

\begin{document}

\maketitle

\begin{abstract}
We consider the vorticity formulation of the Euler equations describing the flow of a two-dimensional incompressible ideal fluid on the sphere. Zeitlin's model provides a finite-dimensional approximation of the vorticity formulation that preserves the underlying geometric structure: it consists of an isospectral Lie–Poisson flow on the Lie algebra of skew-Hermitian matrices. We propose an approximation of Zeitlin's model based on a time-dependent low-rank factorization of the vorticity matrix and evolve a basis of eigenvectors according to the Euler equations. In particular, we show that the approximate flow remains isospectral and Lie--Poisson and that the error in the solution, in the approximation of the Hamiltonian and of the Casimir functions only depends on the approximation of the vorticity matrix at the initial time. The computational complexity of solving the approximate model is shown to scale quadratically with the order of the vorticity matrix and linearly if a further approximation of the stream function is introduced.
\end{abstract}

\section{Introduction}
The motion of inviscid ideal fluids is governed by the Euler equations
which, for incompressible flows, read:
\begin{equation}\label{eq:IE}
\left\{\begin{aligned}
         & \partial_t \ubf +\DIV(\ubf\otimes \ubf) + \nabla p = 0,\\
         & \DIV \ubf = 0,
     \end{aligned}\right.
 \end{equation}
where $\ubf$ represents the velocity of the fluid, and $p$ is the hydrodynamic pressure.
In his pioneering work, Arnold \cite{Arnold66} showed that ideal fluid motions describe geodesics on the Lie group of volume-preserving diffeomorphisms endowed with a right-invariant metric corresponding to kinetic energy.
This result has not only brought to light the geometric structure underlying Euler's equations but it has been used to give rigorous local well-posedness results \cite{EM70}, and to relate the stability of the fluid motion to the sectional curvature of the Riemannian metric \cite{Arnold66, Pre02}, see also \cite[Introduction]{MP24} for further references.
More in general, geodesic motion on a Lie group can be associated with the (basic) \emph{Euler--Poincar\'e equations}, for a purely quadratic Lagrangian, on the corresponding Lie algebra \cite{Sch22}. Other than the incompressible Euler equations \eqref{eq:IE}, many partial differential equations has been shown to fit this framework, although for different infinite-dimensional groups and Riemannian metrics. In the context of hydrodynamics, these models are referred to as \emph{Euler--Arnold equations} and include the Korteweg--de Vries equation, the Camassa--Holm equation, the Landau--Lifshitz equation, the magnetohydrodynamic equations, etc.

When looking at finite-dimensional approximations of the Euler equations, the traditional approach of considering the dynamical variables by Fourier transforming the system, and then truncating at some frequency breaks the geometric structure of the problem, see, e.g., \cite{KSD98}, and leads to unphysical numerical simulations.
To retain as much as possible of the geometric structure of the Euler equations to the finite-dimensional approximation,
numerical methods have been derived in several works \cite{cotter07,pavlov11,gawlik11,GGB21}.

In this work we focus on the two-dimensional incompressible Euler equations on the sphere, relevant for geophysical flows, and leverage the finite-dimensional approximation introduced by Zeitlin \cite{zeitlin91, zeitlin04}. The Zeitlin model is, to the best of our knowledge, the only (spatial) finite-dimensional approximation of the two-dimensional Euler equations, on the torus and on the sphere, that fully adopts Arnold’s geometric description. 
The Zeitlin model consists of an isospectral Lie--Poisson flow for the vorticity matrix on the Lie algebra of skew-Hermitian $N\times N$ matrices.
Local convergence of the solutions of the Zeitlin model to the solutions of the Euler equations, as $N\to\infty$, was first established by Gallagher \cite{G02} and more recently in \cite{FPV22,MV24}. Furthermore, it has been shown in \cite{MP24} that the Zeitlin model also preserves the stable/unstable nature of stationary solutions of the Euler equations.

The fact that the Zeitlin model retains the geometric structure of the Euler equations -- in the sense that it also describes geodesics on a Lie group with a right-invariant Riemannian metric --
results in a coherent approach to the simulation of the qualitative long-time behavior of 2D Euler equations, as it has been recently discovered by Modin and co-authors \cite{MV20b,MV22,cifani2023efficient}. In particular, numerical simulations based on Zeitlin's model have been shown to reproduce the spectral power laws in the inverse energy
cascade \cite{cifani22} and to ensure conservation of Casimir invariants, such as enstrophy, which is critical for 2D turbulence.
The bulk of numerical simulations based on Zeitlin's model rely on a family of numerical time integration schemes, introduced in \cite{MV20}, that preserve the isospectrality and Lie--Poisson structure of the flow. A major bottleneck of this family of time integrators is their computational complexity: the most used second order time integrator of this family scales as $N^3$, even when an efficient $O(N^2)$ computation of the stream function is considered \cite{cifani2023efficient}.

In this work we propose a numerical approximation of the Zeitlin model that preserves the geometric structure of the problem, as derived by Arnold, at a favorable computational complexity. The idea is to perform a time-dependent low-rank factorization of the vorticity matrix and evolve a basis of eigenvectors according to the reconstruction equation of the Euler equations.
In particular, we show that the approximate flow remains isospectral and Lie--Poisson and that the error in the approximation of the Hamiltonian and of the Casimir functions only depends on the approximation of the vorticity matrix at the initial time.
Moreover, we establish a priori error estimates showing that the error, in the Frobenius norm, between the solution of the Zeitlin model and the proposed low-rank approximation is bounded by the truncated singular values of the initial vorticity matrix.
For the numerical time integration of the reconstruction equation, we propose a family of explicit isospectral methods based on Lie groups acting on manifolds. Since these schemes are generally not Lie--Poisson, a second order method based on the implicit midpoint rule is also presented.
The computational complexity of computing the low-rank approximation is shown to scale quadratically with $N$ and linearly if a further approximation of the stream function is introduced.
The proposed method has superior properties in terms of efficiency and accuracy whenever the dynamics has a low-rank structure, for example in the presence of vortex blobs. If this is not the case and the rank of the vorticity matrix equals $N$, then the 
proposed method provides an isospectral, in some cases Lie--Poisson, time integration scheme whose performances are comparable to solving the Zeitlin model with the numerical method of \cite{MV20}.
We also propose an extension of the low-rank approximation to general Euler--Arnold equations characterized by non-isospectral flows. For this alternative approach, a structure-preserving time splitting is introduced to solve the evolution equations for the low-rank factors.

The remainder of the paper is organized as follows.
In \Cref{sec:background} we recall the Zeitlin truncation of the incompressible Euler equations on the sphere and the second order isospectral Lie--Poisson time integrator proposed in \cite{MV20}.
\Cref{sec:lowrank} pertains to the derivation of a low-rank approximation of the Zeitlin model and to the discussion of its geometric properties and convergence results. The resulting approximate dynamics is solved by evolving a basis of eigenvectors, as described in
\Cref{sec:Sfixed}, where a time discretization on the manifold of unitary matrices is presented.
In \Cref{sec:appot} a truncation of the stream matrix is analyzed with the aim of further reducing the computational complexity of the approximate model.
The extension to general Euler--Arnold equations with a factorization of the solution matrix with time-dependent factors is introduced in \Cref{sec:splitting}.
Numerical experiments are discussed in \Cref{sec:numexp}.
\Cref{sec:concl} presents some concluding remarks.

\section{Background}
\label{sec:background}

In this section we recall the vorticity formulation of the 2D incompressible Euler equations and its characterization as an infinite-dimensional Lie--Poisson system on the space of smooth zero-mean functions. We then present the approximation given by the Zeitlin model, a finite-dimensional Lie--Poisson system on the dual of the Lie algebra of traceless skew-Hermitian matrices. The content of the next two sections is largely based on \cite{MR13,MV24}.

\subsection{Vorticity formulation of the 2D incompressible Euler equations}

Let $\Ical:=(0,T]\subset\mathbb{R}$ be a given temporal interval, with $T\in\mathbb{R}_+$.
We consider as spatial domain the unit sphere $\Stwo\subset\mathbb{R}^3$.
Let $\omega:\Ical\times \Stwo\rightarrow\mathbb{R}$ be the vorticity defined as
$\omega=\curl \ubf=\sDIV \ubf$ where $\sDIV$ denotes the skew-divergence. Problem \eqref{eq:IE} in the vorticity variable reads
$\partial_t \omega=-\ubf\cdot\grad\omega$.
Introducing the stream function $\psi:\Ical\times \Stwo\rightarrow\mathbb{R}$ as the skew-gradient of the velocity,
$\sgrad\psi= \ubf$, the vorticity satisfies
\begin{equation}\label{eq:IEvort}
    \left\{\begin{aligned}
        &\partial_t{\omega}=-\grad\omega\cdot\sgrad \psi& \mbox{in }\,\Ical\times\Stwo,\\
        &\Delta\psi = \omega & \mbox{in }\,\Ical\times\Stwo,\\
    \end{aligned}\right.
\end{equation}
where the stream function $\psi$ is related to the vorticity via the Laplace--Beltrami operator $\Delta$.
We remark that, in the inclination-azimuthal spherical coordinates $(\theta,\phi)\in [0,\pi]\times [0,2\pi)$, the gradient and skew-gradient of a function $f(\theta,\phi)\in C^\infty(\Stwo)$, and the skew-divergence of a vector-valued function
$\ubf=u_\theta(\theta,\phi)\ebf_{\theta} + u_\phi(\theta,\phi)\ebf_{\phi}\in C^\infty(\Stwo,\mathbb{R}^2)$ are given by
\begin{equation*}
\begin{aligned}
\grad f(\theta,\phi) &= \dfrac{\partial f}{\partial \theta}\ebf_{\theta} + \dfrac{1}{\sin\theta}\dfrac{\partial f}{\partial \phi}\ebf_{\phi},\\
\sgrad f(\theta,\phi) &= -\dfrac{1}{\sin\theta}\dfrac{\partial f}{\partial \phi}\ebf_{\theta} + \dfrac{\partial f}{\partial \theta}\ebf_{\phi},\\
\sDIV\ubf(\theta,\phi) &= -\dfrac{1}{\sin\theta}\dfrac{\partial u_\theta}{\partial \phi} + \dfrac{1}{\sin\theta}\dfrac{\partial(u_\phi\sin\theta)}{\partial \theta}.
\end{aligned}
\end{equation*}
where $\ebf_{\theta}$ and $\ebf_{\phi}$ are the unit vectors in the inclination/colatitudinal direction and azimuth/longitudinal direction, respectively.

The unit sphere $\Stwo$ is a symplectic manifold with the symplectic form given by the spherical area $\mu$, which,
in spherical coordinates,
reads $\mu=\sin\theta \exd\theta\wedge\exd\phi$.
Let $\mathbb{J}:T\Stwo\to T\Stwo$ represents the rotation by $\pi/2$ in the positive direction.
Given a function $f$ on $\Stwo$, the corresponding Hamiltonian vector field $X_f$ is defined by
$\contr_{X_f}\mu=\exd f$.
By identifying 1-forms with vector fields via the Riemannian structure of $\Stwo$, one can write $\mathbb{J}X_f=\nabla f$. Since the velocity field $\ubf$ of the Euler flow satisfies $\ubf=\sgrad\psi=-\mathbb{J}\nabla\psi$, it is the Hamiltonian vector field for 
the stream function $\psi$, i.e., $\ubf=X_{\psi}$.
The space $\mathcal{X}_{\mu}(\Stwo)$ of smooth Hamiltonian vector fields on $\Stwo$ forms an infinite-dimensional Lie algebra with the vector field bracket 
$[X_{\psi},X_{\xi}]_{\mathcal{X}}=-X_{\psi}\cdot \mathbb{J} X_{\xi}$.
This algebra is isomorphic to the Poisson algebra of smooth functions modulo constants $C^{\infty}(\Stwo)/\mathbb{R}$ via the mapping $\psi\mapsto X_{\psi}$ since
$-[X_{\psi},X_{\xi}]_{\mathcal{X}}=X_{\bra{\psi}{\xi}}$
where $\bra{\psi}{\xi}=\sgrad\psi\cdot\grad\xi$.

The Euler equations \eqref{eq:IEvort} can be characterized as
a Hamiltonian system on the dual of
the Lie algebra $\mathfrak{g}:=C^{\infty}(\Stwo)/\mathbb{R}$. More in details, the smooth dual of $\mathfrak{g}$ can be constructed so that $C^{\infty}(\Stwo)^*\simeq C^{\infty}(\Stwo)$ via the $L^2$-pairing
$$\langle\omega,\psi\rangle_{L^2(\Stwo)}:=\int_{\Stwo} \omega\psi\,\mu.$$
The vorticity can then be thought of as the dual variable to the stream function $\psi\in\mathfrak{g}$:
$$\omega\in\mathfrak{g}^*=\big(C^{\infty}(\Stwo)/\mathbb{R}\big)^*\simeq C_0^{\infty}(\Stwo)=\left\{f\in C^{\infty}(\Stwo)\,:\,\int_{\Stwo} f\,\mu = 0\right\}.$$
The Hamiltonian system on $\mathfrak{g}^*=C_0^{\infty}(\Stwo)$ for a Hamiltonian $H:\mathfrak{g}^*\to\mathbb{R}$ is given by
\begin{equation}\label{eq:LP}
\partial_t \omega +\ad^*_{\delta_{\omega}H}\omega=0,
\end{equation}
where $\ad^*_f:\mathfrak{g}^*\to\mathfrak{g}^*$, $f\in\mathfrak{g}$,
is the representation of the Lie algebra $\mathfrak{g}$ on 
$\mathfrak{g}^*$ defined as
$\langle\ad^*_f\omega,\xi\rangle_{L^2(\Stwo)}=\langle\omega,\bra{f}{\xi}\rangle_{L^2(\Stwo)}$ for any $\xi\in\mathfrak{g}$.
By the divergence theorem, it holds
$$\langle\omega,\bra{f}{\xi}\rangle_{L^2(\Stwo)}=
\int_{\Stwo}\omega\sgrad f\cdot\grad \xi\, \mu=
\langle -\nabla\cdot(\omega\sgrad f),\xi\rangle_{L^2(\Stwo)}=
\langle-\bra{f}{\omega},\xi\rangle_{L^2(\Stwo)}.$$
Hence, $\ad^*_f\omega=-\bra{f}{\omega}$ and being minus the Poisson bracket, it reflects the fact that the $L^2$ pairing on $C^{\infty}(\Stwo)$ is bi-invariant.

Equation \eqref{eq:IEvort} can then be written as the infinite-dimensional Lie--Poisson system
\begin{equation}\label{eq:IELieP}
    \left\{\begin{aligned}
        &\partial_t{\omega}+\bra{\psi}{\omega}=0& \mbox{in }\,\Ical\times\Stwo,\\
        &\Delta\psi = \omega & \mbox{in }\,\Ical\times\Stwo,\\
    \end{aligned}\right.
\end{equation}
with Hamiltonian given by the kinetic energy
\begin{equation*}
    H(\omega)=\frac12 \int_{\Stwo} |\ubf|^2\, \mu=-\frac12 \int_{\Stwo} \psi\,\omega\,\mu = -\frac12 \langle\psi,\omega\rangle_{L^2(\Stwo)},\qquad \delta_{\omega} H(\omega)=-\psi,
\end{equation*}
and where the differential operator $\Delta$ can be thought of as an isomorphism between $\mathfrak{g}=C^{\infty}(\Stwo)/\mathbb{R}$ and its dual $\mathfrak{g}^*=C_0^{\infty}(\Stwo)$.

Since
the vorticity function $\omega$ is infinitesimally transported by the time-dependent Hamiltonian vector field $\ubf=X_{\psi}$, the flow map $\Phi:\Stwo\times\mathbb{R}\to\Stwo$ generates curves $t\mapsto \Phi(\cdot,t)$ in the Lie group corresponding to the Lie algebra $\mathfrak{g}$ and given by the set $\Diff_{\mu}(\Stwo)$ of symplectic diffeomorphisms of $\Stwo$. Thus, any smooth solution $\omega$
of \eqref{eq:IELieP}, with initial condition $\omega_0\in\mathfrak{g}^*$, evolves on the co-adjoint orbit
$$\mathcal{O}(\omega_0)=\{\Ad^*_{\Phi^{-1}}\omega_0=\Phi_*\omega_0\,:\,\Phi\in\Diff_{\mu}(\Stwo)\},$$
where the co-adjoint action $\Ad^*_{\Phi^{-1}}$
of $G=\Diff_{\mu}(\Stwo)$ on $\mathfrak{g}^*=C_0^{\infty}(\Stwo)$
is defined by
$$\inprodL{\Ad^*_{\Phi^{-1}}\omega}{\xi}=
\inprodL{\omega}{\Ad_{\Phi^{-1}}\xi}\qquad \forall\,\Phi\in G,\omega\in\mathfrak{g}^*, \xi\in\mathfrak{g},$$
and the adjoint action $\Ad$ of $G$ on its Lie algebra is the push-forward operation on vector fields, see \cite[Section 14.1]{MR13} for further details.

Any functional of the form
\begin{equation*}
    \mathcal{C}_f(\omega)=\int_{\Stwo} f(\omega)\,\mu,\qquad f:\mathbb{R}\to\mathbb{R},
\end{equation*}
is a conserved quantity of motion, in particular a Casimir invariant.
Linear, quadratic, etc. invariants are obtained by
taking $f(\omega)$ as monomials.
The presence of these integrals imposes an infinite number of constraints on the dynamical variables given by the Fourier components of the vorticity field \cite{zeitlin91}.
A crude truncation in Fourier space would hinder the geometric structure of the problem yielding an inconsistent approximation of the dynamics which induces, in turn, spurious effects and poorly accurate solutions.
This problem was solved by Zeitlin \cite{zeitlin91,zeitlin04} via the so-called sine truncation and based on quantization results of Hoppe \cite{hoppe89}.
The idea is to generate a sequence of finite-mode approximations yielding a finite-dimensional Lie--Poisson system
and providing a number of Casimir functions which tend to the original ones when the truncation size $N$ tends to infinity.

\subsection{The Zeitlin model on $\Stwo$}\label{sec:zeitlin}

The approximation proposed by Zeitlin relies on quantization theory \cite{hoppe89} and $L_{\alpha}$-limits \cite[Definition 2.1]{BMS94} to spatially discretize the vorticity equation \eqref{eq:IELieP} by replacing the infinite-dimensional Poisson algebra of smooth functions with the matrix Lie algebra $\fru(N)$.

The idea \cite{HoppeYau98,BHSS91,BMS94} is to, first, introduce a sequence of linear surjective maps $\{p_N\}_{N\in\mathbb{N}^+}$, $p_N:C^{\infty}(\Stwo)\to\fru(N)$, mapping the unit function $x\mapsto 1$ to the imaginary identity matrix $iI\in\fru(N)$. This implies that, for any $N\in\mathbb{N}^+$, $p_N$ descends to a map between the quotient Lie algebras $C^{\infty}(\Stwo)/\mathbb{R}$ and $\mathfrak{pu}(N):=\fru(N)/i\mathbb{R}I$. The dual $\mathfrak{pu}^*(N)$ is naturally identified with $\mathfrak{su}(N)$ via the pairing corresponding to the Frobenius inner product $\inprodF{W}{P}=\Tr(W^*P)$, with $W\in\mathfrak{su}(N)$ and $P\in\mathfrak{pu}(N)$. The corresponding $\ad^*$-operator on $\mathfrak{su}(N)$ is
$$\ad^*_P(W)=-[P,W]_N:=-\dfrac{1}{\hbar_N}[P,W],\qquad \hbar_N:=\dfrac{2}{\sqrt{N^2-1}},$$
where $[\cdot,\cdot]$ is the matrix commutator.

An explicit expression for the maps $p_N$ can be given by considering the $L^2(\Stwo)$-orthonormal basis for $C^{\infty}_0(\Stwo)$ provided by the complex spherical harmonics
$\{Y_{\ell,m}\,:\,0\leq \ell\leq N-1,\, -\ell\leq m\leq \ell\}$
in inclination-azimuthal coordinates.
For a fixed $N\in\mathbb{N}^+$, the projection $p_N$ can then be defined as
\begin{equation*}
    p_N:\;\omega(\theta,\phi)=\sum_{\ell=0}^{\infty}\sum_{m=-\ell}^{\ell} \omega^{\ell m}Y_{\ell,m}(\theta,\phi)\longmapsto
    \sum_{\ell=0}^{N-1}\sum_{m=-\ell}^{\ell}\ic \omega^{\ell m} T_{\ell,m}^N=W,
\end{equation*}
by associating to each spherical harmonic $Y_{\ell,m}$ a matrix $T^N_{\ell,m}\in\mathfrak{sl}(N,\mathbb{C})$ defined as
\begin{equation*}
    (T_{\ell,m}^N)_{m_1,m_2}:= \sqrt{\dfrac{N}{4\pi}}(-1)^{s-m_1}\sqrt{2\ell+1}
    \begin{pmatrix}
    s & \ell & s\\
    -m_1 & m & m_2
    \end{pmatrix},\qquad s:=\dfrac{N-1}{2},
\end{equation*}
where the term in parenthesis denotes the Wigner 3j-symbol.
Note that, for any fixed pair $(\ell,m)$, the matrix $T^N_{\ell,m}$ has non-zero entries only on the $-m$th diagonal. Indeed, by the properties of the Wigner 3j-symbol, $(T_{\ell,m}^N)_{m_1,m_2}= 0$ whenever
$-m_1+m+m_2\neq 0$.
Moreover, the matrices $\{T^N_{\ell,m}\}_{\ell,m}$ are orthonormal with respect to the scaled Frobenius inner product defined as
$$\inprodN{A}{B}:=\dfrac{4\pi}{N}\inprodF{A}{B}=\dfrac{4\pi}{N}\Tr(A^*B),\qquad \forall\, A,B\in\mathbb{C}^{N\times N}.$$
Note that the $L^2$-dual $p_N^*:\fru(N)\to \mathfrak{g}$ of $p_N$ is a right inverse of $p_N$.
We refer to \cite[Section 3.1]{MV24} for more details on such quantization.

A suitable ``quantized'' Laplace operator $\Delta_N:\fru(N)\to\mathfrak{su}(N)$ is constructed so that its kernel is $i\mathbb{R}I$ and it descends to a bijective map $\mathfrak{pu}(N)\to\mathfrak{su}(N)$.
Typically, the discrete Laplacian is taken so that it keeps the
spectral properties of the Laplace operator \cite{HoppeYau98}:
\begin{equation}\label{eq:lapl}
    \Delta_N T_{\ell,m}^N = -\ell(\ell+1) T_{\ell,m}^N,\qquad 0\leq\ell\leq N-1,\;-\ell\leq m\leq \ell,
\end{equation}
meaning that $T^N_{\ell,m}$ is eigenvector of $\Delta_N$ with eigenvalue $-\ell(\ell+1)$.
The operator $\Delta_N^{-1}$ is normal since it has a basis of eigenvectors that are orthogonal in the Frobenius norm, and it is symmetric since its eigenvalues are real. Hence, $\Delta_N^{-1}$ is self-adjoint with respect to the Frobenius inner product.

Using this matrix algebra approximation of $(C^{\infty}(\Stwo),\bra{\cdot}{\cdot})$,
the Zeitlin model is the finite-dimensional Lie--Poisson system on $\frg^*=\mathfrak{pu}^*(N)\simeq\mathfrak{su}(N)$ given by 
\begin{equation}\label{eq:meqt}
\left\{\begin{aligned}
     & \dot{W} + [P(W),W]_N=0,\\
     & \Delta_N P = W.
 \end{aligned}\right.
\end{equation}
The Zeitlin equation \eqref{eq:meqt} is itself an Euler--Arnold equation for the Lie group $G=\SU(N)$ with Lie algebra $\frg=\mathfrak{su}(N)$,
and Hamiltonian given by
\begin{equation}\label{eq:ham}
    H(W) = -\frac12 \inprodN{W}{P(W)}.
\end{equation}

The co-adjoint action of $R\in \SU(N)$ on $W\in\mathfrak{su}(N)$ is given by $\Ad^*_R(W)=RWR^*$. By the spectral theorem, it follows that the co-adjoint orbit
\begin{equation*}
\mathcal{O}_{\SU(N)}(W_0)
=\{\Ad_R^*(W_0)=R W_0 R^*\,:\, R\in\SU(N)\},\qquad W_0\in\mathfrak{su}(N),
\end{equation*}
consists of all elements of $\mathfrak{su}(N)$ with the same spectrum as $W_0$. Thus, the Zeitlin flow \eqref{eq:meqt} is isospectral since
its solution belongs to the co-adjoint orbit $\mathcal{O}_{\SU(N)}(W_0)$ of its initial condition $W_0\in\mathfrak{su}(N)$.
The quantities
\begin{equation}\label{eq:meqtCas}
    C_k(W):=\dfrac{4\pi}{N}\Tr(W^k),\qquad k=1,\ldots,N,
\end{equation}
are Casimir invariants.
\begin{remark}
The above characterization implies that
there exists $t\mapsto R(t)\in\SU(N)$
solution of the reconstruction equation for geodesics on $\SU(N)$, namely
\begin{equation}\label{eq:rec}
    \dot{R}+\dfrac{1}{\hbar_N}P(RW_0R^{-1})R=0,
\end{equation}
with $R(t_0)=I$, and such that $R(t)W_0R^{-1}(t)$ is the solution of the Zeitlin model. The approximation proposed in \Cref{sec:Sfixed} is inspired by this consideration.
\end{remark}

\subsection{Time integration of the Zeitlin model}
\label{sec:isoInt}
Let us consider a uniform partition of the temporal interval $\Ical=(0,T]=\cup_{\tind}\mathcal{I}_{\tind}$ where $\mathcal{I}_{\tind}:=(t_{\tind},t_{\tind+1}]$,  $t_{\tind}:=\tind \dt$ with $0\leq \tind\leq N_t-1$ and $\dt=T/N_t$.
In \cite{MV20} a numerical temporal integrator that preserves the Lie--Poisson and isospectral structure of the flow of \eqref{eq:meqt} was introduced.
In this section we recall a second order integrator from the family of methods introduced in \cite{MV20}, the one that has been mostly used in numerical simulations of the Zeitlin model and its extensions \cite{cifani2023efficient,franken24,franken24b}.

In each temporal subinterval $\mathcal{I}_{\tind}$, with $\tind\geq 0$, given $W_{\tind}$, the method consists in setting $\widetilde{W}^{(0)}=W_{\tind}$ and then computing, for $j=0,1,\ldots$
\begin{equation}\label{eq:Iso2NLit}
    \widetilde{W}^{(j+1)}=W_{\tind} - \frac{\dt}{2}[P(\widetilde{W}^{(j)}),\widetilde{W}^{(j)}]_N+
    \frac{\dt^2}{4\hbar_N^2}P(\widetilde{W}^{(j)})\widetilde{W}^{(j)} P(\widetilde{W}^{(j)}),
\end{equation}
until a certain stopping criterion is satisfied.
Once the update has stopped at the $n_{\iter}^{\tind}$th iteration, one sets $\widetilde{W}=\widetilde{W}^{(n_{\iter}^{\tind})}$ and computes the updated vorticity matrix as
\begin{equation}\label{eq:Iso2upd}
    W_{\tind+1}=\left(I-\frac{\dt}{2\hbar_N}P(\widetilde{W})\right)\widetilde{W}
    \left(I+\frac{\dt}{2\hbar_N}P(\widetilde{W})\right).
\end{equation}

Note that, to the best of our knowledge, this is the lowest order time integrator of this family on $\mathfrak{su}(N)$ (and on $\fru(N)$).

In the next result we establish the arithmetic complexity of the numerical time integration scheme described above.
To this end,
for an efficient computation of the inverse discrete Laplacian $\Delta_N$ one can observe, as in \cite{cifani2023efficient}, that it is a fourth order tensor which can be split into $2N-1$ blocks $\{\Delta^m\}_{m=-(N-1)}^{N-1}$ of size $N-|m|$, for $m=-(N-1),\ldots, N-1$. Then, the computation of the matrix $P$ consists in solving 
a linear system for each $\Delta^m$ with right hand side given by the $m$th diagonal of $W$ and giving the $m$th diagonal of $P$.
Each $\Delta^m$, for $m\geq 0$, is a tridiagonal symmetric matrix of size $N-|m|$, whose definition is given in, e.g., \cite[Equation (13)]{cifani2023efficient},
and $\Delta^{-|m|}=\Delta^{|m|}$ for any $m$.

\begin{proposition}\label{prop:IsoCost}
Let us consider the numerical time integration scheme \eqref{eq:Iso2NLit}-\eqref{eq:Iso2upd} for the approximation of problem \eqref{eq:meqt} on the temporal interval $\Ical_{\tind}=(t_{\tind},t_{\tind+1}]$, $\tind\geq 0$.
The arithmetic complexity of the algorithm is
\begin{equation*}
    O(N^3 n_{\iter}^{\tind}) 
\end{equation*}
where
$n_{\iter}^{\tind}$ is the
number of iterations required by the nonlinear step \eqref{eq:Iso2NLit}.
\end{proposition}
\begin{proof}
As shown in \cite{cifani2023efficient}, the computation of the stream matrix can be performed in $c N^2$ operations, for some constant $c\in\mathbb{R}_+$.
Indeed, the stream matrix $P$ satisfying the Laplacian problem in \eqref{eq:meqt} can be
obtained by solving $N$ linear systems 
$\Delta^m p_m = w_m$ for $m=0,\ldots,N-1$,
where $p_m$ and $w_m$ denote the $m$th diagonals of $P$ and $W$, respectively.
Since each system has size $N-m$ and it is tridiagonal, Thomas algorithm allows a linear cost in the dimension $N-m$.

Moreover, at the $j$th iteration of the nonlinear solver \eqref{eq:Iso2NLit}, the computation of the bracket requires one multiplication of the stream matrix and of the vorticity matrix, while the last term of \eqref{eq:Iso2NLit} requires one further matrix-matrix multiplication.
Hence, two (typically dense) matrix-matrix multiplications, of complexity $O(N^3)$, are needed for each update of $\widetilde{W}^{(j)}$ and $W_{\tind}$, namely $n_{\iter}^{\tind}+1$ times. More precisely, given the stream matrix $P(\widetilde{W}^{(j)})$,
each iteration requires $4N^3+N^2$ operations.
Therefore, the total arithmetic complexity of the algorithm in $\mathcal{I}_{\tind}$ is
$(n_{\iter}^{\tind}+1)\big(4 N^3 + (c+1)N^2\big)$ and the conclusion follows.
\end{proof}

\section{Geometric low-rank approximation of the Zeitlin model}
\label{sec:lowrank}

As described in \Cref{sec:zeitlin}, the solutions of the Zeitlin model belong to the co-adjoint orbits on $\SU(N)$ associated with given initial conditions.
We propose to approximate the solution $W(t)$ of \eqref{eq:meqt} for the initial condition $W_0\in\mathfrak{su}(N)$ with the trajectory of the Zeitlin flow associated with a low-rank approximation $Y_0$ of $W_0$. More in details,
we consider the initial condition $Y_0\in\mathfrak{u}(N)$
obtained by diagonalizing $W_0$ and truncating its spectrum to the $r\leq N$ eigenvalues of largest modulus.
Since, in general, $Y_0$ has non vanishing trace, the trajectory $Y(t)$ satisfying \eqref{eq:meqt} belongs to the co-adjoint orbit $\mathcal{O}_{\U(N)}(Y_0)\subset\fru(N)$ of $\U(N)$, defined as
\begin{equation}\label{eq:orbit}
\mathcal{O}_{\U(N)}(Y_0)
=\{\Ad_R^*(Y_0)=R Y_0 R^*\,:\, R\in\U(N)\},\qquad Y_0\in\mathfrak{u}(N).
\end{equation}
Let us denote the best rank-$r$ approximation of any $W\in\mathbb{C}^{N\times N}$ as $\Wsvd$ so that $Y_0=\svd{W_0}$.
Let $\sigma_1(W_0)\geq\dots\geq \sigma_N(W_0)$ be the singular values of $W_0$ and let $\{\lambda_j(W_0)\}_{j=1}^N$ be the (purely imaginary) eigenvalues of $W_0$ ordered such that $|\lambda_1(W_0)|\geq\dots\geq |\lambda_N(W_0)|$. Note that $|\Im(\lambda_k(W_0))|=\sigma_k(W_0)$ for $1\leq k\leq N$.
The aforementioned choice of $Y_0$, i.e., $Y_0=\svd{W_0}$, gives $\sigma_k(Y_0)=\sigma_k(W_0)$ and $\lambda_k(Y_0)=\lambda_k(W_0)$ for all $1\leq k\leq r$.
It trivially holds that, if $W_0$ has rank $r\leq N$, then the solution $W(t)$ of \eqref{eq:meqt} coincides with the rank-$r$ solution $Y(t)$ at all times $t\in\Ical$.

By construction, the flow of $Y(t)$ is Lie--Poisson on the dual of $\fru(N)$ and isospectral since $Y(t)\in\mathcal{O}_{\U(N)}(Y_0)$ for any $t\in\Ical$.
This implies that the error between the $k$th Casimir evaluated at the solution $W$ of \eqref{eq:meqt} and at the rank-$r$ solution $Y$ only depends on the approximation at the initial time. Indeed,
\begin{equation}\label{eq:consCasWY}
|C_k(W(t))-C_k(Y(t))|=
\dfrac{4\pi}{N}|\Tr(W_0^k)-\Tr(Y_0^k)|=
\dfrac{4\pi}{N}\Bigg|\sum_{j=r+1}^N \lambda^k_j(W_0)\Bigg|.
\end{equation}
Similarly, the Hamiltonian satisfies
$|H(W(t))-H(Y(t))|=|H(W_0)-H(Y_0)|$.
The rationale for introducing a low-rank approximation of the solution trajectories is that, in the presence of spatially localized solutions of the Euler equations, $Y(t)$ is expected to provide a ``good'' approximation (in a sense to be defined) of $W(t)$ at all times, at a reduced computational cost.
The next sections are devoted to the study of these two aspects. First
we analyze the accuracy of the proposed approximation.

\subsection{A priori error estimates}
In this section we derive an a priori bound on the error between the solution $W$ of the Zeitlin model \eqref{eq:meqt} with initial condition $W_0$ and the low-rank approximate solution $Y$ obtained from the initial condition $Y_0=\svd{W_0}$.
To this end, we first prove the Lipschitz continuity of the velocity field of \eqref{eq:meqt} in the Frobenius norm on the co-adjoint orbits of $\U(N)$.

\begin{lemma}\label{lem:LipsXH}
Let $W_0\in\mathfrak{u}(N)$ be given and let $\mathcal{O}_{\U(N)}(W_0)$ be the associated co-adjoint orbit as defined in \eqref{eq:orbit}.
Then, the operator $X_{\Ham}:\fru(N)\rightarrow \fru(N)$ given by $$X_{\Ham}(A):=
\ad^*_{\Delta_N^{-1}A}A=[A,\Delta_N^{-1} A]_N,\qquad \forall A\in\fru(N),$$
is Lipschitz continuous in the Frobenius norm and it holds
$$\normF{X_{\Ham}(A)-X_{\Ham}(B)}\leq 2\hbar_N^{-1}\sqrt{N}\rho(W_0)\normF{A-B},\qquad\forall\, A,B\in\mathcal{O}_{\U(N)}(W_0)$$
where $\rho(W_0)$ denotes the spectral radius of $W_0$.
\end{lemma}
\begin{proof}
Let $\mathcal{D}:=\{\vc(A)\in\mathbb{C}^{N^2}:\,A\in\mathcal{O}_{\U(N)}(W_0)\}$. Let $\beta:\mathcal{D}\rightarrow \mathcal{D}$ be defined as
\begin{equation*}
    \beta(\vc(A))=\vc(X_{\Ham}(A))=\hbar_N^{-1}(I_N\otimes A^{\top}-A\otimes I_N)\vc(\Delta_N^{-1} A).
\end{equation*}
Since the inverse Laplace operator $\Delta_N^{-1}$ is a linear operator, there exists a matrix $L\in\mathbb{R}^{N^2\times N^2}$ such that $\vc(\Delta_N^{-1} A)=L\vc(A)$.
Let $J_{\beta}(\vc(A))\in\mathbb{C}^{N^2\times N^2}$ denote the Jacobian matrix of $\beta$ at $\vc(A)$.
We observe that $\beta\in C^1$ and the map $\vc(A)\mapsto J_{\beta}(\vc(A))$ is continuous. Since $\mathcal{D}$ is compact, $\norm{J_{\beta}(\vc(A))}_2$ attains a maximum value in $\mathcal{D}$.
In particular, for a given $\vc(A)\in\mathbb{C}^{N^2}$, the Jacobian applied to any $\vc(\eta)\in\mathbb{C}^{N^2}$ is given by $J_{\beta}(\vc(A))(\vc(\eta))=
\vc(\ad^*_{\Delta_N^{-1}\eta}A+\ad^*_{\Delta_N^{-1}A}\eta)$.
Observing that $\norm{\Delta_N^{-1}}_2=\frac12$, we thus have,
for any $\vc(A)\in\mathcal{D}$,
\begin{equation*}
\begin{aligned}
    \norm{J_{\beta}(\vc(A))}_2& =\sup_{\norm{\eta}_2=1}\norm{J_{\beta}(\vc(A))(\vc(\eta))}_2\\
& \leq \dfrac{1}{\hbar_N} \normF{\Delta_N^{-1}A} \sup_{\norm{\eta}_2=1}\norm{I\otimes\eta^{\top}-\eta\otimes I}_2
+ \dfrac{2}{\hbar_N}\norm{A}_2 \sup_{\norm{\eta}_2=1}\normF{\Delta_N^{-1}\eta} \\
& \leq  \dfrac{2}{\hbar_N} \norm{A}_F\leq
\dfrac{2}{\hbar_N}\sqrt{N}\norm{A}_2 = \dfrac{2}{\hbar_N}\sqrt{N}\rho(W_0)\sim N^{3/2}\rho(W_0).
\end{aligned}
\end{equation*}
Then, the conclusion follows.
\end{proof}
To derive a priori error estimates between the solution $W$ of the Zeitlin model \eqref{eq:meqt} with initial condition $W_0\in\mathfrak{su}(N)$ and the solution $Y$ of \eqref{eq:meqt} with initial condition $\svd{W_0}$, we derive an error bound between $Y$ and the best low-rank approximation of $W$ at each time.
Since the flow \eqref{eq:meqt} is isospectral, the error between $W(t)$ and its best rank-$r$ approximation $\svd{W(t)}$, which is given by the truncated SVD of $W(t)$ at time $t\in\Ical$, is constant in time: by the Eckart--Young--Mirsky theorem \cite{EYM36}, it holds
\begin{equation*}
    \normF{W(t)-\svd{W(t)}}^2 = \sum_{i=r+1}^{N} \sigma_i^2(W_0).
\end{equation*}
Note that the best rank-$r$ approximation is unique if and only if, for all $t\geq 0$, $\sigma_r(W(t))\neq \sigma_{r+1}(W(t))$. Without loss of generality we can always consider the case in which $r$ is such that $\sigma_r(W_0)> \sigma_{r+1}(W_0)$.

Observe that the best rank-$r$ approximation of $W$ satisfies the evolution equation
\begin{equation}\label{eq:evolSVD}
    d_t(\svd{W(t)}) = [\Wsvd,P(W)]_N=X_{\Ham}(\Wsvd)+[\Wsvd,P(W-\Wsvd)]_N,
\end{equation}
where the first term only depends on the retained modes, while the last term takes into account the interaction of the retained modes with the neglected ones. To better highlight the relationship between the evolution of $\Wsvd$ and the one of the low-rank approximation $Y$, equation \eqref{eq:evolSVD} can be equivalently written as
\begin{equation*}
    d_t(\svd{W(t)}) = \Pi_{T_{\Wsvd}\mathcal{M}_r}(X_{\Ham}(W))+L_{\Wsvd}[W-\Wsvd](X_{\Ham}(W))
\end{equation*}
where $X\mapsto L_{\Wsvd}[N](X)$ is the Weingarten map at $\Wsvd$ with normal direction $N$, and $\Pi_{T_{\Wsvd}\mathcal{M}_r}$ denotes the projection onto the tangent space of the rank-$r$ matrix manifold $\mathcal{M}_r$ at $\Wsvd$. 

\begin{proposition}\label{prop:err}
Let $W_0\in\mathfrak{su}(N)$. For any $t\in(0,T]$, let $W(t)$ be solution of \eqref{eq:meqt} with initial condition $W_0$, let $Y(t)$ be the solution of \eqref{eq:meqt} with initial condition $Y_0=\svd{W_0}$ and 
let $\svd{W(t)}$ be the best rank-$r$ approximation of $W(t)$. Then,
\begin{equation}\label{eq:errYWsvd}
    \normF{Y(t)-\svd{W(t)}}
    \leq
    \dfrac{\normF{\svd{W_0}}}{\hbar_N K}(e^{K t}-1) \sqrt{\sum_{i=r+1}^{N} \sigma_i^2(W_0)}\,,
\end{equation}
with $K$
Lipschitz continuity constant of $X_{\Ham}=-\ad^*_{\nabla H}$ in the Frobenius norm.
\end{proposition}
\begin{proof}
Using the evolution equations \eqref{eq:evolSVD} for $\Wsvd$ and \eqref{eq:meqt} for $Y$ gives
\begin{equation}\label{eq:bound1}
\begin{aligned}
    \normF{\dot{Y}(t)-d_t(\svd{W(t)})}
     \leq &\, \normF{X_{\Ham}(Y)-X_{\Ham}(\svd{W})}\\
     & +\normF{[\Wsvd,P(W-\svd{W})]_N}.
\end{aligned}
\end{equation}
The first term can be bounded using the Lipschitz continuity of the Hamiltonian vector field $X_H$ as shown in \Cref{lem:LipsXH}. 
The second term in \eqref{eq:bound1} can be bounded using the linearity of the commutator as
\begin{equation*}
\begin{aligned}
 \normF{[\Wsvd,P(W-\Wsvd)]_N}
 & \leq 2\hbar_N^{-1}\normF{\Wsvd}\normF{P(W-\Wsvd)}.
\end{aligned}
\end{equation*}
Moreover, $\normF{P(W)}\leq \norm{\Delta_N^{-1}}_2\normF{W}=2^{-1}\normF{W_0}$, for any $t\geq 0$.
Note that, more generally, the norm of the velocity field of the flow satisfies \cite[Theorem 2.2]{WuXu10} 
\begin{equation*}
    \hbar_N^2\normF{X_{\Ham}(W)}^2
    \leq 2\normF{P(W)}^2\normF{W}^2-8 H(W)^2\leq
    \frac12\normF{W_0}^4-8 H(W)^2.
\end{equation*}
Combining the bounds above we get, for any $t\geq 0$,
\begin{equation*}
\begin{aligned}
    \normF{\dot{Y}(t)-d_t(\svd{W(t)})}
    \leq &\, K \normF{Y(t)-\svd{W(t)}}\\
    & + \hbar_N^{-1}\normF{\svd{W(t)}}\normF{W(t)-\svd{W(t)}}.
\end{aligned}
\end{equation*}
Gronwall's inequality together with $\normF{W(t)-\svd{W(t)}}=\normF{W_0-\svd{W_0}}$ and
$\normF{\svd{W(t)}}=\normF{\svd{W_0}}$, gives
\begin{equation*}
    \normF{Y(t)-\svd{W(t)}}
    \leq  \hbar_N^{-1}\normF{\svd{W_0}}\normF{W_0-\svd{W_0}}\int_0^t e^{K(t-s)}\, ds.
\end{equation*}
Computing the integral yields the conclusion.
\end{proof}
Note that, if we take $K=K(N)=2\hbar_N^{-1}\sqrt{N}\rho(W_0)$ as in \Cref{lem:LipsXH} then
the  error in \eqref{eq:errYWsvd} can be bounded as
\begin{equation*}
     \normF{Y(t)-\svd{W(t)}}
     \leq
     \dfrac12(e^{K(N) t}-1) \normF{W_0-\svd{W_0}}\,.
\end{equation*}
Since the Lipschitz constant depends on $N$ the convergence, as $r\to N$, is not uniform in $N$.

The proposed low-rank approximation preserves the geometric structure of the Zeitlin model and it converges to its solution in the sense of \Cref{prop:err}.
However, solving problem \eqref{eq:meqt} from the low-rank initial condition $\svd{W_0}$ using the time integration scheme of \Cref{sec:isoInt} is computationally as expensive as solving the full-rank problem, namely the Zeitlin model with full-rank initial condition $W_0$.
The idea is then to exploit the low-rank structure of the approximate rank-$r$ solution via a suitable factorization.
In the next sections we consider two possible options.

\section{Fixed spectrum approximation}\label{sec:Sfixed}
At the initial time, since $W_0\in\mathfrak{su}(N)$ is normal, it is unitarily diagonalizable.
We then approximate the solution $W(t)\in\mathfrak{su}(N)$ of \eqref{eq:meqt}, at any time $t\in\Ical$, with
\begin{equation}\label{eq:fact}
    Y(t)=U(t)S_0 U^*(t)\in\fru(N),
\end{equation}
where $S_0\in\mathbb{C}^{r\times r}$ contains the $r$ largest eigenvalues of $W_0$, and, for any $t\in\Ical$, $U(t)$
belongs to the Stiefel manifold $\St(r,\mathbb{C}^{N})$.
At the initial time we set $U(t_0)=U_0\in\mathbb{C}^{N\times r}$, where the columns of $U_0$ are the eigenvectors associated with the $r$ largest eigenvalues of $W_0$, and require that $t\mapsto U(t)$ satisfies the reconstruction equation \eqref{eq:rec} on $\St(r,\mathbb{C}^{N})$, that is
    \begin{equation}\label{eq:US0}
\dot{U} + \dfrac{1}{\hbar_N} P(US_0U^*) U = 0.
    \end{equation}
Note that, with the factorization introduced in \eqref{eq:fact}, solving problem \eqref{eq:US0} corresponds to evolving a basis of eigenvectors for the approximate state $Y$.

By construction, the approximate trajectory $t\mapsto Y(t)=U(t)S_0U^*(t)$, obtained by solving \eqref{eq:US0} for $U$, is solution of the Zeitlin model \eqref{eq:meqt} with initial condition $Y_0=U_0S_0U^*_0=\svd{W_0}$ and, thus, it satisfies the geometric properties discussed in \Cref{sec:lowrank}.
In particular, if the rank of the solution $W$ of problem \eqref{eq:meqt} is $r\leq N$, then the approximate state $Y$ in \eqref{eq:fact} coincides with $W$.

For the numerical temporal approximation of problem \eqref{eq:US0} we need to make sure that $U(t)\in\mathbb{C}^{N\times r}$ remains on $\St(r,\mathbb{C}^N)$ for all $t\in\Ical$.
Moreover, we aim at achieving a computational complexity lower than $O(N^3)$, namely the one required to solve the original problem \eqref{eq:meqt} and described in \Cref{prop:IsoCost}.

\subsection{Time integration of the reconstruction equation}
\label{sec:RKMK}

In this section we describe and analyze numerical time integrators for the solution of the evolution equation \eqref{eq:US0}.
The first class of methods we propose is based on Lie groups acting on manifolds. The resulting methods fall within the class of numerical integration schemes known as Runge--Kutta Munthe-Kaas (RK-MK) methods \cite{RKMK}. They are explicit, of arbitrary order, and, although not Lie--Poisson, they allow to preserve the isospectrality of the flow by ensuring that the discrete trajectories remain on the co-adjoint orbits \eqref{eq:orbit}.

Let us first observe that problem \eqref{eq:US0} can be written as $\dot{U}(t) = \Lcal_{S_0}(U(t))\, U(t)$ for some $\Lcal_{S_0}:\St(r,\mathbb{C}^{N})\rightarrow\fru(N)$.
The idea is then to derive an evolution equation on the Lie algebra $\fru(N)$ via a coordinate map (of the first kind), namely a smooth function
$\coo:\fru(N)\rightarrow \U(N)$.
The coordinate map should satisfy $\coo(0)=I\in \U(N)$ and $\dcoo_0=I$, where
$\dcoo:\fru(N)\times\fru(N)\rightarrow\fru(N)$ is the right trivialized tangent of $\coo$ defined as
$d_t \coo(A(t))=\dcoo_{A(t)}(\dot{A}(t))\coo(A(t))$,
for any $A:\mathbb{R}\rightarrow \fru(N)$.
For sufficiently small $t\geq t_0$, the solution of \eqref{eq:US0}
is given by
$U(t) = \coo(\Omega(t))U(t_0)$ where $\Omega(t)\in\fru(N)$
satisfies
\begin{equation}\label{eq:NalgS0}
		\dot{\Omega}(t) = \dcoo_{\Omega(t)}^{-1}\big(\Lcal_{S_0}\big(U(t)\big)\big),\qquad\mbox{for }\; t\in\Ical,
\end{equation}
with $\Omega(t_0) = 0$.
Problem \eqref{eq:NalgS0} can be solved using traditional Runge--Kutta (RK) methods.
Let $(b_i,a_{i,j})$, for $i=1,\ldots,N_s$ and $j=1,\ldots,N_s$,
be the coefficients of the Butcher tableau describing an $N_s$-stage explicit RK method. Then, the numerical approximation of \eqref{eq:NalgS0} in the interval $\mathcal{I}_{\tind}=(t_{\tind},t_{\tind+1}]$, $\tind\geq 0$, is performed as in \Cref{algo:RKMK}.

\begin{algorithm}[H]
\caption{Explicit RK-MK scheme to solve \eqref{eq:NalgS0} in $(t_\tind,t_{\tind+1}]$}\label{algo:RKMK}
\normalsize
\begin{algorithmic}[1]
 \Require{$U_{\tind}\in\St(r,\mathbb{R}^{N})$, $\{b_i\}_{i=1}^{N_s}$, $\{a_{i,j}\}_{i,j=1}^{N_s}$ }
 \State $\Omega_{{\tind}}^1=0$, $U_{\tind}^1=U_{\tind}$
 \For{$i=2,\ldots,N_s$}
 \State $\Omega_{{\tind}}^i = \dt\sum\limits_{j=1}^{i-1} a_{i,j} \, \dcoo_{\Omega_{\tind}^j}^{-1}\big(\Lcal_{S_0}(U_{\tind}^j)\big)$,\label{line:Mi}
 \State $U_{\tind}^i = \coo(\Omega_{{\tind}}^i)\,U_{\tind}$,\label{line:Ui}
 \EndFor
 \State $\Omega_{{\tind}+1} = \dt\sum\limits_{i=1}^{N_s} b_i\,\dcoo_{\Omega_{{\tind}}^i}^{-1}\big(\Lcal_{S_0}(U_{\tind}^i)\big)$,
 \State \Return $U_{{\tind}+1} = \coo(\Omega_{{\tind}+1})\,U_{\tind}\in\St(r,\mathbb{R}^{N})$
\end{algorithmic}
\end{algorithm}
The choices of the coordinate map $\coo$ and of the function 
$\Lcal_{S_0}:\St(r,\mathbb{C}^{N})\rightarrow\fru(N)$ are clearly not unique. In this work, we aim at choosing them such that the arithmetic complexity of \Cref{algo:RKMK} is at most linear in $N$ excluding the computational cost associated with the evaluation of $\Lcal_{S_0}$, which is problem dependent.
To this aim it is crucial to deal with low-rank quantities in the application of both the coordinate map and its tangent inverse, as it will be shown in \Cref{prop:RKMKCost}. 
Although the natural choice for $\Lcal_{S_0}$ would be $\Lcal_{S_0}(U)=-\hbar_N^{-1}P(US_0U^*)$ for any $U\in \St(r,\mathbb{C}^{N})$, this quantity is typically not low-rank. We thus opt for the alternative, yet equivalent, choice $\Lcal_{S_0}(U)=(I-UU^*)\mathcal{F}_{S_0}(U)U^*-U\mathcal{F}_{S_0}^*(U)$ where $\Fcal_{S_0}(U)=-\hbar_N^{-1}P(US_0U^*)U$.
Note that $\Lcal_{S_0}(U)\in\fru(N)$ since $U^*\Fcal_{S_0}(U)\in\fru(r)$
for any $U\in\St(r,\mathbb{C}^N)$.

As coordinate map we consider the Cayley transform
\begin{equation}\label{eq:cay}
    \cay(\Omega):=\Big(I-\dfrac{\Omega}{2}\Big)^{-1}\Big(I+\dfrac{\Omega}{2}\Big).
\end{equation}
\begin{proposition}\label{prop:RKMKCost}
Let us consider the explicit $N_s$-stage RK-MK time integration scheme in \Cref{algo:RKMK} for the approximation of problem \eqref{eq:US0} over the temporal interval $\Ical_{\tind}$, $\tind\geq 0$.
Let $\psi$ be the Cayley transform \eqref{eq:cay}.
The arithmetic complexity is 
\begin{equation}\label{eq:RKMKcompl}
    O(N r^2 N_s^2) + O(r^3 N_s^4) + O(N^2 r N_s).
\end{equation}
\end{proposition}
\begin{proof}
The idea is that one can decompose $\Lcal_{S_0}(U_{\tind})$ into the sum of low-rank factors. Indeed, for any $i=1,\ldots,N_s$,
$$\Lcal_{S_0}(U_{\tind}^i)=(I-U_{\tind}^i(U_{\tind}^i)^*)\mathcal{F}_{S_0}(U_{\tind}^i)(U_{\tind}^i)^*-U_{\tind}^i\mathcal{F}_{S_0}^*(U_{\tind}^i)=
a_ib_i^*-U_{\tind}^ic^*_i$$
where $a_i := [\Fcal_{S_0}(U_{\tind}^i)|-U_{\tind}^i]\in\mathbb{C}^{N\times 2r}$,
$b_i := [U_{\tind}^i|\Fcal_{S_0}(U_{\tind}^i)]\in\mathbb{C}^{N\times 2r}$, and
$c_i := U_{\tind}^i\Fcal_{S_0}^*(U_{\tind}^i)U_{\tind}^i\in\mathbb{C}^{N\times r}$.
This implies that the inverse tangent map of the Cayley transform admits, in turn, a low-rank factorization:
$$
\Lambda_i:=\dcoo_{\Omega_{\tind}^i}^{-1}\big(\Lcal_{S_0}(U_{\tind}^i)\big)
=\dcoo_{\Omega_{\tind}^j}^{-1}(a_ib_i^*-U_{\tind}^ic^*_i)
= A_{\tind}^ia_i(A_{\tind}^ib_i)^*-A_{\tind}^iU_{\tind}^i(A_{\tind}^ic_i)^*=\alpha_i\beta_i^*
$$
where
$\alpha_i:=[A_{\tind}^ia_i|-A_{\tind}^iU_{\tind}^i]\in\mathbb{C}^{N\times 3r}$,
$\beta_i:=[A_{\tind}^ib_i|A_{\tind}^ic_i]\in\mathbb{C}^{N\times 3r}$ and
$A_{\tind}^i:= I-\Omega_{\tind}^i/2$.

The evaluation of $\Fcal_{S_0}(U)$, for any $U\in\mathbb{C}^{N\times r}$, requires the computation of the stream matrix $P(US_0U^*)$ and its multiplication by $U$, for a total complexity of $O(N^2 r)$.
Then, given $\Lcal_{S_0}(U_{\tind}^i)$, the computation of $\alpha_i$ and $\beta_i$ requires $O(Nr\rank(\Omega_{\tind}^i))$ operations.
From the definition of $\Omega_{\tind}^i$ at line 4 of the algorithm and the properties of the matrices $\Lambda_i$, one has that $k_i:=\rank(\Omega_{\tind}^i)\leq 3r(i-1)$.
This implies that the computation of $U_{\tind}^i$ at line 5 requires $O(Nrk_i)+O(k_i^2r)+O(k_i^3)$, as shown in, e.g., \cite[Proposition 5.2]{P21}.
The conclusion follows by summing these quantities.
\end{proof}

\begin{remark}
In this work we focus on the complexity reduction in the number of degrees of freedom $N$ and consider numerical time integration schemes of order at most $2$, i.e., with $N_s=2$.
For higher order timestepping, the polynomial complexity in the number $N_s$ of stages in \eqref{eq:RKMKcompl} can be mitigated by using tangent methods as the one proposed in \cite{CO02}.
\end{remark}

\subsubsection{Properties of the approximate solution}
A bound on the approximation error $\normF{W(t_{\tind})-Y_{\tind}}$, for any $\tind$, can be obtained by combining the bound \eqref{eq:errYWsvd} between the exact solution $Y$ and the best rank-$r$ approximation of $W$ with the following result.
\begin{lemma}
Let $Y_{\tind}=U_{\tind}S_0U_{\tind}^*$ be the rank-$r$ approximate solution at time $t_{\tind}$ with $U_{\tind}$ obtained from a RK-MK method of order $p$ as in \Cref{algo:RKMK}.
Let $Y(t_{\tind})=U(t_{\tind})S_0 U^*(t_{\tind})$ with $U(t_{\tind})$ exact solution of \eqref{eq:US0} at time $t_{\tind}$.
Then, there exists a positive constant $c\in\mathbb{R}$ such that
$$\normF{Y_{\tind}-Y(t_{\tind})}\leq c\dt^p \sqrt{\sum_{i=1}^r\sigma_i^2(W_0)}.$$
\end{lemma}
\begin{proof}
By simply applying the triangle inequality, one gets
$$\normF{Y_{\tind}-Y(t_{\tind})}
\leq \normF{U_{\tind}-U(t_{\tind})}\normF{S_0}\Big(\norm{U(t_{\tind})}_2+\norm{U_{\tind}}_2\Big),$$
and $\normF{S_0}^2=\normF{Y_0}^2=\sum_{i=1}^r\sigma_i^2(W_0)$.
\end{proof}

The discrete flow of $Y_{\tind}$ is isospectral and the Casimir functions \eqref{eq:meqtCas} satisfy,
$$|C_k(W(t_{\tind}))-C_k(Y_{\tind})|=\dfrac{4\pi}{N}\Bigg|\sum_{j=r+1}^N \lambda^k_j(W_0)\Bigg|,\qquad \forall\, 1\leq k\leq N.$$
Indeed, since the approximation $U_{\tind}$ of $U(t_{\tind})$ obtained from \Cref{algo:RKMK} belongs, by construction, to $\St(r,\mathbb{C}^{N})$ for any $\tind$, the approximate solution $Y_{\tind}=U_{\tind}S_0U_{\tind}^*$ belongs to $\fru(N)$.
As a consequence, the discrete flow of $Y_{\tind}$ belongs to the co-adjoint orbit $\mathcal{O}_{\U(N)}(Y_0)$ and the Casimir functions \eqref{eq:meqtCas} are exactly preserved, that is,
$|C_k(Y_{\tind})-C_k(Y_0)|=0$ for any $\tind$ and any $1\leq k\leq N$.

\subsection{A second order Lie--Poisson isospectral scheme}\label{sec:midU}
The family of methods described in \Cref{algo:RKMK} provides numerical time integration schemes that are isospectral, in the sense that $Y_{\tau}\in\mathcal{O}_{\U(N)}(Y_0)$ for any $\tau\geq 0$, but are not Lie--Poisson for the problem \eqref{eq:meqt} with initial condition $Y_0$.

A second order isospectral and Lie--Poisson time integration scheme for \eqref{eq:meqt} can be derived by discretizing the reconstruction equation \eqref{eq:US0} with the implicit midpoint rule.
\begin{proposition}
  Given $W_0\in\mathfrak{su}(N)$, let $Y_0=\svd{W_0}$ and let $U_0\in\St(r,\mathbb{C}^N)$ contain the eigenvectors of $W_0$ associated with its $r$ largest eigenvalues.
  The implicit midpoint rule applied to the reconstruction equation \eqref{eq:US0} in  the temporal interval $\Ical_{\tind}=(t_{\tind},t_{\tind+1}]$, namely
  \begin{equation}\label{eq:midU}
  U_{\tind+1}-U_{\tind}=-\hbar_N^{-1} \dt P(\widetilde{U} S_0 \widetilde{U}^*)\widetilde{U}\qquad\widetilde{U}:=\dfrac{U_{\tind+1}+U_{\tind}}{2},\quad \forall \tau\geq 0,
  \end{equation}
descends to an isospectral Lie--Poisson integrator for the Zeitlin model \eqref{eq:meqt} with initial condition $Y_0$.
In particular, the numerical time integration scheme
$$Y_{\tind}=U_{\tind}S_0 U^*_{\tind}\,\longmapsto\, Y_{\tind+1}=U_{\tind+1}S_0 U^*_{\tind+1}, \qquad\forall \tau\geq 0,$$
coincides with the second order isospectral Lie--Poisson integrator of
\cite[Definition 1]{MV20} described in \eqref{eq:Iso2NLit}-\eqref{eq:Iso2upd}.
\end{proposition}
\begin{proof}
Given $U_{\tind}\in\St(r,\mathbb{C}^N)$, for any $\tind\geq 0$, the numerical scheme \eqref{eq:midU} can be written as
\begin{equation}\label{eq:IsoU}
\left\{\begin{aligned}
    & U_{\tind} = \left(I+\frac{\dt}{2\hbar_N}P(\widetilde{U}S_0\widetilde{U}^*)\right)\widetilde{U},\\
    & U_{\tind+1} = \left(I-\frac{\dt}{2\hbar_N}P(\widetilde{U}S_0\widetilde{U}^*)\right)\widetilde{U}.
\end{aligned}\right.
\end{equation}
Then, $U_{\tind+1}$ belongs to $\St(r,\mathbb{C}^N)$.
Introducing the auxiliary variable $\widetilde{Y}:=\widetilde{U}S_0\widetilde{U}^*\in\mathbb{C}^{N\times N}$, and replacing in $Y_{\tind+1}$ the expression of $U_{\tind+1}$ given in \eqref{eq:IsoU}, results in
$$Y_{\tind+1}=U_{\tind+1}S_0 U^*_{\tind+1}=
\left(I-\frac{\dt}{2\hbar_N}P(\widetilde{Y})\right)\widetilde{Y}\left(I+\frac{\dt}{2\hbar_N}P(\widetilde{Y})\right).$$
Similarly, replacing in $Y_{\tind}$ the expression of $U_{\tind}$ given in \eqref{eq:IsoU}, results in
$$Y_{\tind}=U_{\tind}S_0 U^*_{\tind}=
\left(I+\frac{\dt}{2\hbar_N}P(\widetilde{Y})\right)\widetilde{Y}\left(I-\frac{\dt}{2\hbar_N}P(\widetilde{Y})\right).$$
The above equations for $Y_{\tind}$ and $Y_{\tind+1}$ coincide with \cite[Eq. (2.3)]{Viviani20}, namely the second order isospectral Lie--Poisson integrator of \cite{MV20}.
\end{proof}
The practical implementation of problem \eqref{eq:IsoU}
in each temporal subinterval $\mathcal{I}_{\tind}$, with $\tind\geq 0$, given $U_{\tind}\in\St(r,\mathbb{C}^N)$, is analogous to the one presented in \Cref{sec:isoInt}. More in detail, the method consists in setting $\widetilde{U}^{(0)}=U_{\tind}$ and then computing
\begin{equation}\label{eq:Iso2UNLit}
    \widetilde{U}^{(j+1)}=U_{\tind} - \frac{\dt}{2\hbar_N}P(\widetilde{U}^{(j)}S_0(\widetilde{U}^{(j)})^*)\widetilde{U}^{(j)},\qquad j=0,1,\ldots
\end{equation}
until a certain stopping criterion is satisfied.
Once the update has stopped at the $n_{\iter}^{\tind}$th iteration, one sets $\widetilde{U}=\widetilde{U}^{(n_{\iter}^{\tind})}$ and computes the updated basis as
\begin{equation}\label{eq:Iso2Uupd}
    U_{\tind+1}=\left(I-\frac{\dt}{2\hbar_N}P(\widetilde{U}S_0\widetilde{U}^*)\right)\widetilde{U}.
\end{equation}
Proceeding as in \Cref{prop:IsoCost}, it is easy to verify that the arithmetic complexity to solve problem \eqref{eq:midU} with the above algorithm \eqref{eq:Iso2UNLit}-\eqref{eq:Iso2Uupd} is
$O(N^2 r n_{\iter}^{\tind})$, 
where $n_{\iter}^{\tind}$
it is the number of iterations required by the nonlinear step in the temporal interval $\Ical_{\tind}$.
Note that the numerical solution obtained with the time integration scheme \eqref{eq:IsoU} and implemented as above does not yield the same solution of the method in \Cref{sec:isoInt} because of the different nonlinear iteration.

\section{Approximation of the stream function}\label{sec:appot}

The computational cost of solving the discrete Laplace equation in \eqref{eq:meqt} for the stream matrix $P$ is the bottleneck of the algorithm, as shown in the proof of \Cref{prop:RKMKCost}.
One possibility to speed up its computation is to evaluate the stream function on an approximation of the vorticity matrix as follows.

Let $0<\ap{N}\leq N-1$; we introduce the operator $\apT:\mathbb{C}^{N\times N}\rightarrow\mathbb{C}^{N\times N}$ that, when applied to $W$, sets to zero the $j$th diagonal of $W$ for any $|j|>\ap{N}$.
When considering matrices obtained as the image of the projection maps $p_N$ from \Cref{sec:zeitlin}, $\apT$ can be written as
\begin{equation}\label{eq:Top}
    \apT:\; W=\sum_{\ell=0}^{N-1}\sum_{m=-\ell}^{\ell}\ic\omega^{\ell m} T_{\ell,m}^N\;\longmapsto\;
    \sum_{\ell=0}^{N-1}\sum_{m=-\min\{\ell,\ap{N}\}}^{\min\{\ell,\ap{N}\}}\ic\omega^{\ell m} T_{\ell,m}^N.
\end{equation}
We point out that this truncation strategy has already appeared in \cite{CEV24}, where only the large-scale components of the dynamics of the 2D Euler equations are considered and a closure term is used to model their interaction with the small scales.
Using this approximation in the Laplace equation in \eqref{eq:meqt} yields the following evolution equation:
\begin{equation}\label{eq:meqt_hypred}
\left\{\begin{aligned}
    & \dot{Z} + [\ap{P}(Z),Z]_N = 0,\\
    & \Delta_N \ap{P} = \apT(Z).
\end{aligned}\right.
\end{equation}

To show that the flow \eqref{eq:meqt_hypred} is isospectral and Lie--Poisson we need the following technical, yet elementary, result.
\begin{lemma}\label{lem:Top}
    Let the operator $\apT:\fru(N)\rightarrow\fru(N)$ be defined as in \eqref{eq:Top} for some $0<\ap{N}\leq N-1$.
    Then, $\apT$ is linear, self-adjoint with respect to the Frobenius inner product and it commutes with $\Delta_{N}^{-1}$, the inverse of the Laplace operator from \eqref{eq:lapl}.
\end{lemma}
Owing to the properties of $\apT$ from \Cref{lem:Top} and to the self-adjointness of $\Delta_N$ with respect to the Frobenius inner product, problem \eqref{eq:meqt_hypred} is still isospectral 
 and Lie--Poisson, but with an approximate Hamiltonian,
resulting from the truncation introduced by \eqref{eq:Top}, that is
\begin{equation}\label{eq:apHam}
    \ap{\Ham}(Z) = -\frac12 \inprodN{Z}{\ap{P}(Z)}.
\end{equation}
The error in the approximation of the Hamiltonian can be bounded by the error in the approximation of the initial condition and by the Hamiltonian approximation at the initial time, as follows.
For any $t\in\Ical$, let $W(t)$ be the solution of \eqref{eq:meqt} with initial condition $W_0$ and let $Z(t)$ be solution of \eqref{eq:meqt_hypred} with initial condition $Z_0$; then
$$|\Ham(W(t))-\ap{\Ham}(Z(t))|=|\Ham(W_0)-\ap{\Ham}(Z_0)|
\leq |\Ham(W_0)-\Ham(Z_0)|+|\Ham(Z_0)-\ap{\Ham}(Z_0)|.$$
The approximation introduced by $\apT$ has no effect on the Casimir invariants.
\begin{remark}
    When $Z_0=W_0$, the error in the approximation of the Hamiltonian is only due to the initial approximation error $|\Ham(W_0)-\ap{\Ham}(W_0)|$.    
\end{remark}
Concerning the accuracy of the approximation, we can establish an error bound analogous to \Cref{prop:err} with a further term that depends on the truncation \eqref{eq:Top} and goes to zero as $\ap{N}$ tends to $N$.
\begin{proposition}\label{prop:errT}
Let $W_0\in\mathfrak{su}(N)$. For any $t\in(0,T]$, let $W(t)$ be solution of \eqref{eq:meqt} with initial condition $W_0$, let $Z(t)$ be the solution of \eqref{eq:meqt_hypred} with initial condition $Z_0=\svd{W_0}$ and 
let $\svd{W(t)}$ be the best rank-$r$ approximation of $W(t)$. Then,
\begin{equation}\label{eq:errZWsvd}
\begin{aligned}
    \normF{Z(t)-\svd{W(t)}}
    \leq &\,
    \left(1 +\dfrac{\normF{W_0}}{\hbar_N\ap{K}}\right) (e^{\ap{K} t}-1) \sqrt{\sum_{i=r+1}^{N} \sigma_i^2(W_0)}\\
   & + \normF{W_0}\int_0^t \normF{P(W(s))-\ap{P}(W(s))}e^{\ap{K} (t-s)}\,ds,
   \end{aligned}
\end{equation}
with $\ap{K}$ Lipschitz continuity constant of
$X_{\ap{\Ham}}=-\ad^*_{\nabla \ap{H}}$
in the Frobenius norm.
\end{proposition}
\begin{proof}
The reasoning is analogous to the proof of \Cref{prop:err}; here we need to consider the extra term associated with the approximation of $\Ham$ by $\ap{\Ham}$.

Using the evolution equations \eqref{eq:evolSVD} for $\Wsvd$ and \eqref{eq:meqt_hypred} for $Z$ gives
\begin{equation}\label{eq:bound1T}
\begin{aligned}
    \normF{\dot{Z}-d_t(\svd{W(t)})}
     =\,& \normF{X_{\ap{\Ham}}(Z)-[\Wsvd,P(W)]_N}
    \leq \normF{X_{\ap{\Ham}}(W)-X_{\ap{\Ham}}(Z)}\\
    & + \normF{X_{\Ham}(W)-[\Wsvd,P(W)]_N}
    + \normF{X_{\Ham}(W)-X_{\ap{\Ham}}(W)}.
\end{aligned}
\end{equation}
The first two terms can be bounded as in the proof of \Cref{prop:err} using the Lipschitz continuity of the Hamiltonian vector field $X_{\ap{\Ham}}$ and the linearity of the matrix commutator, resulting in
\begin{equation*}
\begin{aligned}
\normF{X_{\ap{\Ham}}(W)-X_{\ap{\Ham}}(Z)}+ \normF{X_{\Ham}(W)-[\Wsvd,P(W)]_N}
& \leq \ap{K} \normF{W-Z}\\
+\hbar_N^{-1} \normF{W_0}
\normF{W-\Wsvd}.
\end{aligned}
\end{equation*}
The last term in \eqref{eq:bound1T} can be bounded as
\begin{equation*}
    \normF{X_{\Ham}(W)-X_{\ap{\Ham}}(W)}
    = \normF{[W,P(W)-\ap{P}(W)]_N}
    \leq \normF{P(W)-\ap{P}(W)}\normF{W_0}.
\end{equation*}
Combining the bounds above we get, for any $t\geq 0$,
\begin{equation*}
\begin{aligned}
    \normF{\dot{Z}-d_t(\svd{W(t)})}
    \leq &\, (\ap{K}+\hbar_N^{-1}\normF{W_0})\normF{W-\Wsvd} + \ap{K}\normF{\Wsvd-Z}\\
    & +\normF{W_0} \normF{P(W)-\ap{P}(W)}.
\end{aligned}
\end{equation*}
Using Gronwall's inequality and the fact that $\normF{W(t)-\svd{W(t)}}=\normF{W_0-\svd{W_0}}$ yields the conclusion.
\end{proof}
Note that the last term of \eqref{eq:errZWsvd}
can be further bounded as
\begin{equation*}
    \int_0^t \normF{P(W(s))-\ap{P}(W(s))}e^{\ap{K} (t-s)}\,ds
    \leq \dfrac12\int_0^t \normF{W(s)-\apT(W(s))}e^{\ap{K} (t-s)}\,ds.
\end{equation*}

\subsection{Computational complexity of the approximate dynamics}

The approximation of the stream function introduced in \eqref{eq:Top} can be applied to both the Zeitlin model \eqref{eq:meqt} and to the reconstruction equation \eqref{eq:rec}. In both cases the computational complexity is lowered by a factor $N$ leading to a complexity quadratic in $N$ for the full-rank Zeitlin model and linear in $N$ when solving the reconstruction equation, as shown in the next results.

\begin{proposition}
Let us consider the numerical time integration scheme \eqref{eq:Iso2NLit}-\eqref{eq:Iso2upd} for the approximation of problem \eqref{eq:meqt_hypred} over the temporal interval $\Ical_{\tind}$, $\tind\geq 0$.
The arithmetic complexity of the algorithm is
$O(N^2\ap{N} n_{\iter}^{\tind})$
where $n_{\iter}^{\tind}$ is the
number of iterations required by the nonlinear step \eqref{eq:Iso2NLit}.
\end{proposition}
\begin{proof}
Repeating the steps of the proof of \Cref{prop:IsoCost}, one has that
the stream matrix $\ap{P}$ satisfies an approximate Laplacian problem and it can be computed by solving $\ap{N}+1$ linear systems 
$\Delta^m p_m = z_m$ for $m=0,\ldots,\ap{N}$,
where $p_m$ and $z_m$ denote the $m$th diagonals of $P$ and $Z$, respectively.
Since each system has size $N-m$ and it is tridiagonal, the number of operations required is
$\sum_{m=0}^{\ap{N}} (N-m)$, thus leading arithmetic complexity $O(N\ap{N})$.

Moreover, at the $j$th iteration of the nonlinear solver \eqref{eq:Iso2NLit}, the computation of the bracket involves two matrix-matrix multiplications that depend on the sparse stream matrix. This requires $O(N^2\ap{N})$
operations for each update of $\widetilde{W}^{(j)}$ and $W_{\tind}$, namely $n_{\iter}^{\tind}+1$ times.
\end{proof}

The approximation of the stream function \eqref{eq:Top} can be combined to the factorization of the state proposed in \Cref{sec:Sfixed} leading to the evolution equation: given $U(t_{0})=U_0\in\St(r,\mathbb{C}^{N})$, find $U(t)\in\St(r,\mathbb{C}^N)$ such that
\begin{equation}\label{eq:apUS0}
    \left\{
    \begin{aligned}
        & \dot{U} + \dfrac{1}{\hbar_N} \ap{P}(US_0U^*) U=0,\\
        & \Delta_N \ap{P} = \apT(US_0U^*).
        \end{aligned}\right.
 \end{equation}
\begin{proposition}
    Let us consider the $N_s$-stage explicit RK-MK time integration scheme in \Cref{algo:RKMK} for the approximation of problem \eqref{eq:apUS0} over the temporal interval $\Ical_{\tind}$, $\tind\geq 0$.
Let $\psi$ be the Cayley transform \eqref{eq:cay}.
The arithmetic complexity is 
\begin{equation*}
    O(N r^2 N_s^2) + O(r^3 N_s^4) + O(N\ap{N} r N_s).
\end{equation*}
\end{proposition}
\begin{proof}
The reasoning is analogous to the proof of \Cref{prop:RKMKCost}. The only part that changes is the cost to evaluate the term
$\ap{\Fcal}(U):=\ap{P}(US_0U)U$ for any $U\in\mathbb{C}^{N\times r}$.

To solve the approximate stream matrix one needs to reconstruct the $m$th diagonals of the state $Y:=US_0U$ only for $0\leq m\leq \ap{N}$. The arithmetic complexity of computing the $m$th diagonal of $Y$ is $O(r(N-m))$.
The total cost to assemble the right-hand side of the approximate Laplace equation is thus $r\sum_{m=0}^{\ap{N}} (N-m)$, i.e.,
$O(N\ap{N}r)$.
The solution of the approximate Laplace equation is then $O(N\ap{N})$ and the matrix-matrix multiplication costs $O(Nr)$ owing to the sparsity of $\ap{P}$.

The result follows by combining this cost with the arithmetic complexities derived in the proof of \Cref{prop:RKMKCost} for the intermediate steps of the RK-MK time integrator.
\end{proof}

Using a similar argument as in the proof above one can show the following result for the implicit midpoint rule applied to the reconstruction equation on $\St(r,\mathbb{C}^N)$, as described in \Cref{sec:midU}.
\begin{proposition}
Let us consider the implicit midpoint rule as in \eqref{eq:Iso2UNLit}-\eqref{eq:Iso2Uupd} for the approximation of problem \eqref{eq:apUS0} over the temporal interval $\Ical_{\tind}$, $\tind\geq 0$.
The arithmetic complexity of the algorithm is
$O(N\ap{N} r n_{\iter}^{\tind})$
where $n_{\iter}^{\tind}$ is the
number of iterations required by the nonlinear step \eqref{eq:Iso2UNLit}.
\end{proposition}


\section{Approximate dynamics via splitting}\label{sec:splitting}

Another possibility to perform a low-rank approximation of problem \eqref{eq:meqt} is to consider the factorization of the approximate state, at any time $t\in\Ical$, given by
\begin{equation}\label{eq:factSt}
    Y(t)=U(t)S(t) U^*(t).
\end{equation}
At the initial time we set $U(t_0)=U_0\in\mathbb{C}^{N\times r}$ where the columns of $U_0$ are the eigenvectors associated with the $r$ largest eigenvalues of $W_0$,
and $S_0\in\mathbb{C}^{r\times r}$ contains the $r$ largest eigenvalues of $W_0$.
Moreover, we require that, for any $t\in\Ical$,
$U(t)$ belongs to the Stiefel manifold $\St(r,\mathbb{C}^{N})$
and $S(t)\in\mathbb{C}^{r\times r}$ is skew-Hermitian.
These two conditions ensure that $Y(t)\in\fru(N)$ for any $t\in\Ical$.
Moreover, if $U$ belongs to $\St(r,\mathbb{C}^{N})$, then $\Tr(Y^k)=\Tr(S^k)$ for any $k\geq 1$.
Note that, differently from the factorization \eqref{eq:fact}, we allow both $U$ and $S$ to vary in time.
The factorization \eqref{eq:factSt} is less preferable than \eqref{eq:fact} when dealing with isospectral flows but it is suitable for more general flows on $\fru(N)$ and it can be easily adapted to other matrix algebras.

To derive evolution equations for the factors $U$ and $S$ in \eqref{eq:factSt}, we propose the following decomposition of the velocity field
\begin{equation}\label{eq:decomp}
    X_{\Ham}(Y)=\underbrace{\Pi_{\Rcal(U)}X_{\Ham}(Y)\Pi_{\Rcal^{\perp}(U)}+\Pi_{\Rcal^{\perp}(U)}X_{\Ham}(Y)\Pi_{\Rcal(U)}}_{=:\Pi_U(X_{\Ham}(Y))}
    + \underbrace{\Pi_{\Rcal(U)}X_{\Ham}(Y)\Pi_{\Rcal(U)}}_{=:\Pi_S(X_{\Ham}(Y))},
\end{equation}
where $\Pi_{\Rcal(U)}$ is the orthogonal projection onto the range of $U$.
Note that such decomposition holds for the orthogonal, with respect to the Frobenius norm, projection onto $T_Y\mathcal{M}_r$ of any vector field, that is
$\Pi_{T_Y\mathcal{M}_r}X=\Pi_U(X)+\Pi_S(X)$ for any $X\in\mathbb{C}^{N\times N}$ and $Y\in\Mr$.
Exploiting the decomposition of $X_{\Ham}$ introduced above we can split the evolution equation for $Y$ into
the following evolution equations:
\begin{align}
         & \dot{S}=U^*X_{H}(USU^*)U=-[U^*P(USU^*)U,S]_N,\label{eq:Scont}\\
         & \dot{U}+\dfrac{1}{\hbar_N}(I-UU^*)P(USU^*)U = 0.\label{eq:Ucont}
\end{align}
This system retains the geometric properties of the Zeitlin model, as shown in the next results.
\begin{proposition}\label{prop:LPIsoROMS}
Let $\mathbb{U}\in\St(r,\mathbb{C}^N)$ be fixed. Then, problem \eqref{eq:Scont} with $U(t)=\mathbb{U}$, for any $t\in\Ical$, is isospectral and Lie--Poisson on the dual of $\fru(r)$ with Hamiltonian 
\begin{equation*}
    H_{\mathbb{U}}(S) := H(\mathbb{U}S\mathbb{U}^*)=-\frac12 \inprodN{\mathbb{U}S\mathbb{U}^*}{P(\mathbb{U}S\mathbb{U}^*)}
\end{equation*}
\end{proposition}
Moreover, the evolution of $U$ in \eqref{eq:Ucont} remains on the Stiefel manifold and the Hamiltonian is a conserved quantity whenever $S$ is fixed, as shown in the following result.
\begin{proposition}\label{prop:orthU}
Let $U(t)$ be solution of \eqref{eq:Ucont} in $\Ical$ with initial condition $U_0\in\St(r,\mathbb{C}^{N})$. Then, $U(t)$ belongs to $\St(r,\mathbb{C}^{N})$ for any $t\in\Ical$. Moreover, if $Y=U\mathbb{S}U^*$ with $\mathbb{S}\in\fru(r)$ fixed, the Hamiltonian $H(Y)$ defined in \eqref{eq:ham} is conserved.
\end{proposition}
\begin{proof}
To show that the solution of \eqref{eq:Ucont} remains on the Stiefel manifold $\St(r,\mathbb{C}^{N})$ one can simply verify that
$d_t(U^*(t)U(t))=0$ for all $t$.
Since $U_0\in\St(r,\mathbb{C}^{N})$ by assumption, the conclusion follows.

To show the conservation of the Hamiltonian, we first write the evolution equation for $Y=U\mathbb{S}U^*$ based on \eqref{eq:Ucont}; thereby
$$\dot{Y}=\dot{U}\mathbb{S}U^*+U\mathbb{S}\dot{U}^*
=-[P(Y),Y]_N+U[U^*P(Y)U,\mathbb{S}]_N U^*.$$
Using this expression for $\dot{Y}$ 
yields
\begin{equation*}
\begin{aligned}
    d_t H(Y)
    & = \inprodN{P(Y)}{[P(Y),Y]_N}- \inprodN{P(Y)}{U[U^*P(Y)U,\mathbb{S}]_NU^*}\\
    & = -\inprodN{U^*P(Y)U}{[U^*P(Y)U,\mathbb{S}]_N}=0.
    \end{aligned}
    \end{equation*}
\end{proof}
Note that results analogous to \Cref{prop:LPIsoROMS,prop:orthU} hold if we introduce in \eqref{eq:Scont}-\eqref{eq:Ucont} the approximation of the stream function $P$ from \Cref{sec:appot} with Hamiltonian given by 
$\ap{H}_{\mathbb{U}}(S):=\ap{H}(\mathbb{U}S\mathbb{U}^*)$
where $\ap{H}$ is defined in \eqref{eq:apHam}.

\begin{corollary}
Let $t\in\Ical$ be fixed. Assume $Y(t)=U(t)S(t)U^*(t)$ is obtained from $S(t)$, solution of \eqref{eq:Scont} and $U(t)$, solution of \eqref{eq:Ucont}. Then, $Y(t)$ belongs to $\fru(N)$ for all $t\in\Ical$. Moreover, the Casimir functions \eqref{eq:meqtCas} satisfy
$C_k(Y(t))=C_k(S(t))$, for any $k\geq 1$, and are, thus, conserved quantities.
\end{corollary}

\subsection{Time integration of the approximate dynamics via splitting}
Let $\Phi_U$ and $\Phi_S$ denote the flux associated with the projection operator $\Pi_U$ and $\Pi_S$ from \eqref{eq:decomp}, respectively, so that $\Phi_U(t_{\tind},t_{\tind+1},Y(t_{\tind}))$ is solution of $\dot{Y} = \Pi_U(\dot{Y})$ in $\Ical_{\tind}=(t_{\tind},t_{\tind+1}]$, with $\tind\geq 0$ and similarly for $\Phi_S$.
Problem \eqref{eq:Scont}-\eqref{eq:Ucont} can be solved using a splitting scheme. We focus on a second order consistent splitting, such as the Strang splitting, since the lowest order isospectral Lie--Poisson integrator of \cite{MV20} has order $2$.

The Strang splitting integrator in the time interval $\Ical_{\tind}$ reads
\begin{equation*}
Y_{\tind+1}=\Phi_U\Big(t_{\tind+1/2},t_{\tind+1},\Phi_S\big(t_{\tind},t_{\tind+1},\Phi_U(t_{\tind},t_{\tind+1/2},Y_{\tind})\big)\Big).
\end{equation*}
Given $Y(t_{\tind})=U(t_{\tind})S(t_{\tind}) U^*(t_{\tind})$, perform the following steps in $\Ical_{\tind}$.
\begin{itemize}
    \item Let $\mathbb{S} = S(t_{\tind})$. Starting from the initial condition $U(t_{\tind})$, compute $U(t_{\tind+1/2})$ by solving the $N\times r$ problem
    \begin{equation}\label{eq:U}
        \dot{U}(t) +\hbar_N^{-1} (I-U(t)U^*(t))P(U(t)\mathbb{S}U^*(t))U(t)\qquad \mbox{for}\;t\in(t_{\tind},t_{\tind+1/2}].
    \end{equation}
    \item Let $\mathbb{U} = U(t_{\tind+1/2})$.
    Starting from the initial condition $S(t_{\tind})$, 
    compute $S(t_{\tind+1})$ by solving the $r\times r$ problem 
    \begin{equation}\label{eq:S}
        \dot{S}(t) 
        = -[\mathbb{U}^*P(\mathbb{U}S(t)\mathbb{U}^*) \mathbb{U},S(t)]_N\qquad \mbox{for}\;t\in(t_{\tind},t_{\tind+1}].
    \end{equation}
    \item Starting from the initial condition $U(t_{\tind+1/2})$, compute $U(t_{\tind+1})$ by solving 
    problem \eqref{eq:U} in $(t_{\tind+1/2},t_{\tind+1}]$,
    with $\mathbb{S} = S(t_{\tind+1})$.
\end{itemize}
Then, return $Y(t_{\tind+1})=U(t_{\tind+1})S(t_{\tind+1})U^*(t_{\tind+1})$. 

The Hamiltonian is conserved by the splitting owing to \Cref{prop:LPIsoROMS,prop:orthU}.
Taking $\mathbb{U}=U(t_{\tind+1/2})$,
$\mathbb{S}_{\tind}=S(t_{\tind})$, and
$\mathbb{S}_{\tind+1}=S(t_{\tind+1})$ results in
\begin{equation*}
    \begin{aligned}
        H(Y(t_{\tind+1}))&
        =H(U(t_{\tind+1/2})\mathbb{S}_{\tind+1} U^*(t_{\tind+1/2}))
=H_{\mathbb{U}}(S(t_{\tind+1}))=H_{\mathbb{U}}(S(t_{\tind}))\\
& = H(U(t_{\tind+1/2})\mathbb{S}_{\tind}U^*(t_{\tind+1/2}))
=H(U(t_{\tind})\mathbb{S}_{\tind}U^*(t_{\tind}))
=H(Y(t_{\tind})).
    \end{aligned}
\end{equation*}
Similarly, the Casimir functions are preserved since the factor $S$ satisfies an isospectral flow and the factor $U$ is in the Stiefel manifold.
\begin{remark}
In principle one could solve problem \eqref{eq:Scont}-\eqref{eq:Ucont} via the DLRA splitting integrator introduced in \cite{LO14} or one of its extensions.
However, the conservation of the Casimirs is not guaranteed: for one, integrating the evolution equation \eqref{eq:Scont} for $S$ from the initial condition $S(t_{\tind})=(U^*(t_{\tind+1})U(t_{\tind})) S(t_{\tind})(U^*(t_{\tind+1})U(t_{\tind}))^*$ prevents the flow of $S$ from being isospectral. Moreover, the Hamiltonian is, in general, no longer globally conserved.
\end{remark}

The evolution equation for the low-dimensional factor $S$ in \eqref{eq:Scont} can be solved using the isospectral and Lie--Poisson integrator proposed in \cite{MV20}.
The time integrator is as in \eqref{eq:Iso2NLit}-\eqref{eq:Iso2upd} with the stream function $P$ replaced by $U^*PU$.
The evolution equation for the factor $U$ can be solved using one of the time integrators of \Cref{sec:RKMK} and \Cref{sec:midU}.
By construction, these numerical time integrators combined with the splitting ensure that $S_{\tind}\in\fru(r)$ and $U_{\tind}\in\St(r,\mathbb{C}^N)$ for any $\tind\geq 0$, which gives
$Y_{\tind}=U_{\tind}S_{\tind}U^*_{\tind}\in\fru(N)$ for any $\tind\geq 0$.
Moreover, the Casimir functions are preserved by the temporal discretization.

\subsection{Computational complexity of solving the approximate dynamics}
The computational complexity of solving the evolution equation for $S$, in the temporal subinterval $\Ical_{\tind}$, with the isospectral scheme of \Cref{sec:isoInt} is $O(N^2 r n_{\iter}^{\tind})+O(r^3 n_{\iter}^{\tind})$ where $n_{\iter}^{\tind}$ is the
number of iterations required by the nonlinear step \eqref{eq:Iso2NLit}.
This can be easily verified by reproducing the steps of the proof of \Cref{prop:IsoCost}.
The evolution of the factor $U$ in \eqref{eq:Ucont} solved with a RK-MK time integrator has the arithmetic complexity proven in \Cref{prop:RKMKCost}, namely
$O(N r^2 N_s^2) + O(r^3 N_s^4) + O(N^2 r N_s)$.
The proof is analogous but with $\Lcal_{S}(U):=\Fcal_{S}(U)U^*-U\Fcal_{S}^*(U)$ and $\Fcal_{S}(U)=-\hbar_N^{-1}(I-UU^*)P(USU^*)U$.
This means that, contrary to the solver for the original model, see \Cref{prop:IsoCost}, the proposed low-rank splitting scales quadratically with $N$.

When the stream function $P$ is approximated as in \Cref{sec:appot}, then the computational complexity of the algorithm in the time interval $\Ical_{\tind}$ reduces to
$$O(N\ap{N} r^2 n_{\iter}^{\tind})+O(r^3 n_{\iter}^{\tind})+
O(N r^2 N_s^2) + O(r^3 N_s^4) + O(N\ap{N} r N_s).$$

\section{Numerical experiments}\label{sec:numexp}
We test the performances of the proposed methods on two test cases.
In the first one we consider as initial condition a random skew-Hermitian matrix with a prescribed spectrum. In the second test case we consider a more physical simulation of vortex blob dynamics.

Concerning the notation, if not otherwise specified, the symbol $\Wref$ will denote a reference solution, $\svd{\Wref}$ the truncated SVD of $\Wref$, and $\Omega$ a generic numerical solution.
In the legend we will use the shorthand ``Zeitlin'' to refer to the Zeitlin model \eqref{eq:meqt} with initial condition $W_0\in\mathfrak{su}(N)$, ``Rec($r$)'' to refer to the reconstruction equation \eqref{eq:US0} associated with the initial condition $Y_0=\svd{W_0}$ of rank $r$, and ``TRec($r$,$\ap{N}$)'' to refer to the reconstruction equation with truncation \eqref{eq:apUS0} associated with the rank-$r$ initial condition $Z_0=Y_0=\svd{W_0}$ and truncation to the $\ap{N}$th mode.

We compare different time integrators:
$\Iso$ refers to the second order isospectral Lie--Poisson integrator of \cite{MV20} and summarized in \Cref{sec:isoInt};
RKMK-p, with p $\in\{1,2\}$, refers to the $p$th order Runge--Kutta Munthe-Kaas time integrator described in \Cref{sec:RKMK} for the solution of the evolution equation \eqref{eq:US0} (or \eqref{eq:U}) for the factor $U$ in the low-rank factorization \eqref{eq:fact} (or \eqref{eq:factSt}).

\subsection{Random initial condition with prescribed spectrum}
Given $N$, we consider as initial condition a matrix $W_0\in\fru(N)$ generated randomly from a standard normal distribution but with prescribed spectrum: the eigenvalues are taken such that the absolute value of their imaginary part ranges uniformly in the interval $[10^{-11},10]$.
We assess the performances of the different algorithms by comparison with a reference solution $\Wref$ obtained by solving the Zeitlin model \eqref{eq:meqt} with initial condition $W_0$, $N$ degrees of freedom, and the second order isospectral solver $\Iso$ with a fine time step $\dt$ that will be specified case by case. We compare the different algorithms on the temporal interval $[0,T=1]$.

\subsubsection{Full-rank approximation $r=N$, different time steps $\dt$}
As a first test we set $N=100$ and compute the reference solution $\Wref$ using $\Iso$ with $\dt=10^{-6}$.
We take the approximation rank $r$ equal to the problem dimension $N$ and let $\dt$ vary. This allows us to assess the performances of the numerical time integration schemes and of the factorization \eqref{eq:fact} without considering the error introduced by a low-rank approximation.
In \Cref{fig:errvsdtcpu} we show the error in the Frobenius norm between the reference solution $\Wref(T)$ at final time $T=1$ and an approximation $\Omega(T)$. We compare the cases where $\Omega$ is obtained by solving the Zeitlin model \eqref{eq:meqt} with the $\Iso$ scheme and the full-rank reconstruction equation \eqref{eq:US0} with the RK-MK time integrator of order one and two.
In \Cref{fig:errvsdtcpu} (left) we observe that, as expected, all numerical solutions converge to the reference one with the order of the corresponding scheme. We also record that, at least in this test case, the error obtained with the reconstruction equation (red line with stars) is lower than the one obtained with the Zeitlin model and $\Iso$ solver (blue line with dots).
In \Cref{fig:errvsdtcpu} (right) we report the error between the reference and the approximate solutions versus the algorithm runtime (in seconds). Each datum refers to a different value of $\dt$.
We observe that the second order RK-MK solver (red line with stars) is computationally more efficient than the $\Iso$ solver (blue line with dots), for a given time step $\dt$. As an example an error of the order of $10^{-9}$ is achieved with the second order approximate solver 8 times faster than with the $\Iso$ solver for the Zeitlin model.

\begin{figure}[H]
\centering
\begin{tikzpicture}
    \begin{groupplot}[
      group style={group size=2 by 1,
                  horizontal sep=.5cm},
      width=6.5cm, height=5.25cm
    ]
    \nextgroupplot[xlabel={$\dt$},
                 ylabel={$\norm{\Wref(T)-\Omega(T)}$},
                 axis line style = thick,
                 grid=both,
                 minor tick num=2,
                 grid style = {gray,opacity=0.2},
                 xmode=log, ymode=log,
                 xmin = 5e-4, xmax = 1e-1,
                 ymax = 1e-4, ymin = 1e-12,
                 every axis plot/.append style={thick},
                 legend cell align=left,
                 ylabel near ticks,
                 xlabel style={font=\footnotesize},
                 ylabel style={font=\footnotesize},
                 x tick label style={font=\footnotesize},
                 y tick label style={font=\footnotesize},
                 legend style={font=\scriptsize},
                 legend columns=5, 
                 legend style={at={(.95,1.2)},anchor=north}]
        \addplot+[mark=diamond,color=cyan] table[x=dt,y=errRefYf] {data/errvsdt_N100.txt}; 
        \addplot+[mark=star,color=red] table[x=dt,y=errRefY] {data/errvsdt_N100.txt};
        \addplot+[mark=*,color=blue] table[x=dt,y=errRefW] {data/errvsdt_N100.txt};
        \addplot+[mark=none,color=black, dashed] table[x=dt,y=h] {data/errvsdt_N100.txt};
        \addplot+[mark=none,color=black, dotted] table[x=dt,y=htwo] {data/errvsdt_N100.txt};
        \legend{Rec($r{=}N$) RKMK-1$\quad$,
                Rec($r{=}N$) RKMK-2$\quad$,
                Zeitlin $\Iso\quad$,
                $\dt\quad$,
                $\dt^2$};
        \nextgroupplot[xlabel={Runtime $[s]$},
                 axis line style = thick,
                 grid=both,
                 minor tick num=2,
                 grid style = {gray,opacity=0.2},
                 ymode=log,
                 ymin = 1e-12, ymax = 1e-4,
                 xmin = 0, xmax = 30,
                 every axis plot/.append style={thick},
                 xlabel style={font=\footnotesize},
                 x tick label style={font=\footnotesize},
                 yticklabel=\empty]
        \addplot+[mark=diamond,color=cyan] table[x=cpuYf,y=errRefYf] {data/errvscpu_N100.txt}; 
        \addplot+[mark=star,color=red] table[x=cpuY,y=errRefY] {data/errvscpu_N100.txt};
        \addplot+[mark=*, color=blue,
        every node near coord/.append style={xshift=1cm},
            every node near coord/.append style={yshift=-0cm},
            nodes near coords, 
            point meta=explicit symbolic,
            every node near coord/.append style={font=\footnotesize}]
            table[x=cpuW,y=errRefW, meta index = 7] {data/errvscpu_N100.txt};
    \end{groupplot}
\end{tikzpicture}
\vspace{-1em}
\caption{Error, at the final time $T$, between the reference solution and different approximate solutions vs. the time step $\dt$ (left) and vs. the algorithm runtime (right). On the right plot each datum refers to a different value of $\dt$.}
\label{fig:errvsdtcpu}
\end{figure}
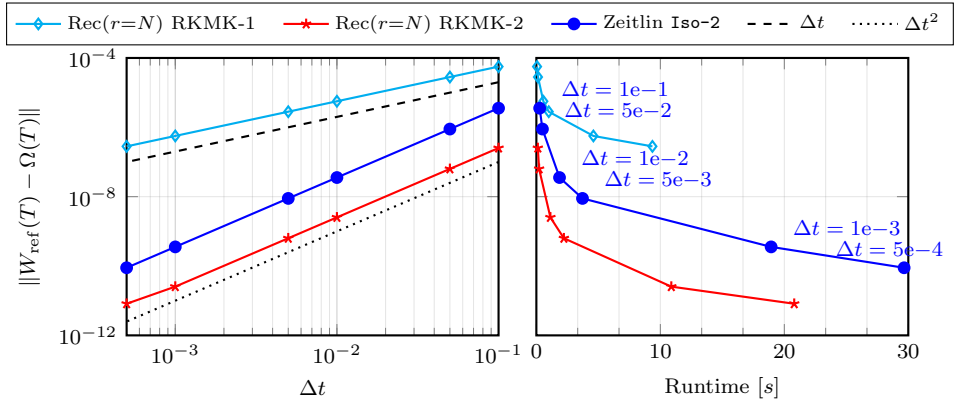

\subsubsection{Fixed time step $\dt$, different approximation ranks $r$}
In this test case we fix $N=500$ and compute the reference solution $\Wref$ using the second order isospectral integrator $\Iso$ with $\dt=10^{-5}$. We want to compare the approximate solution also with the best low-rank approximation $\svd{\Wref}$ given by the truncated SVD of $\Wref$ at each time.

As time integrator for the evolution equation \eqref{eq:US0} of the factor $U$ in \eqref{eq:fact} we consider the second order RK-MK scheme with $\dt=10^{-2}$.

In \Cref{fig:ROMerrvsr} we report the error, in the Frobenius norm and at the final time, $\normF{Y(T)-\Wref(T)}$ between the rank-$r$ solution and the reference solution and the error $\normF{Y(T)-\svd{\Wref(T)}}$ between the rank-$r$ solution and the best low-rank approximation of the reference solution vs. the approximation rank $r$.
We can infer that the distance between the approximate trajectory and the best rank-$r$ approximation is lower than the approximation error due to the low-rank truncation. 
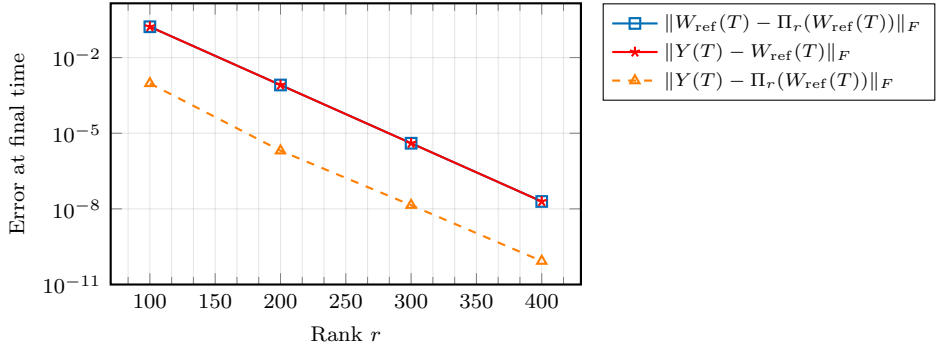
\begin{figure}[H]
\centering
\begin{tikzpicture}
    \begin{axis}[width = 7.8cm, height = 5.3cm,
                 xlabel={Rank $r$},
                 ylabel={Error at final time},
                 axis line style = thick,
                 grid=both,
                 minor tick num=2,
                 grid style = {gray,opacity=0.2},
                 ymode=log,
                 every axis plot/.append style={thick},
                 legend cell align=left,
                 ylabel near ticks,
                 xlabel style={font=\footnotesize},
                 ylabel style={font=\footnotesize},
                 x tick label style={font=\footnotesize},
                 y tick label style={font=\footnotesize},
                 legend style={font=\scriptsize},
                 legend style={at={(1.4,1)},anchor=north},
                 mark options={solid}]
        \addplot+[mark=square,color=NavyBlue] table[x=ROMsize,y=errRefWsvd] {data/errvsROMsize_N500.txt};
        \addplot+[mark=star,color=red] table[x=ROMsize,y=errRefY] {data/errvsROMsize_N500.txt}; 
        \addplot+[mark=triangle,color=orange,dashed] table[x=ROMsize,y=errWsvdY] {data/errvsROMsize_N500.txt};
        \legend{$\normF{\Wref(T)-\svd{\Wref(T)}}$,
                $\normF{Y(T)-\Wref(T)}$,
                $\normF{Y(T)-\svd{\Wref(T)}}$};
    \end{axis}
\end{tikzpicture}
\caption{Approximation errors at the final time vs. the approximation rank $r$.}
\label{fig:ROMerrvsr}
\end{figure}

To assess the performances of the low-rank approximation in terms of computational efficiency, we report in \Cref{fig:ROMerrvsCPU} the error between the reference solution and the approximate solution versus the algorithm runtime. The blue line with dots refers to the solution of the Zeitlin model \eqref{eq:meqt} using the $\Iso$ scheme,
while the red line with crosses refers to the low-rank approximation \eqref{eq:fact}.
We observe that, although the computational cost of solving the low-rank model increases with its size $r$, as expected, it is always lower than the cost required to solve the Zeitlin model.
In particular, from the left plot we can infer that the rank-$r$ approximation can achieve a more accurate solution at a lower computational cost. 

\begin{figure}[H]
\centering
\begin{tikzpicture}
    \begin{groupplot}[
      group style={group size=2 by 1,
                  horizontal sep=1.2cm},
      width=6.5cm, height=5.2cm
    ]
    \nextgroupplot[ylabel={$\normF{\Wref(T)-\Omega(T)}$},
                  xlabel={Runtime [s]},
                  axis line style = thick,
                  grid=both,
                  minor tick num=3,
                  grid style = {gray,opacity=0.2},
                  xmin = 14, xmax = 34,
                  xtick = {14,18,22,26,30,34},
                  ytick={1e-11,1e-9,1e-7,1e-5,1e-3,1e-1},
                  ymode=log,
                  xlabel style={font=\footnotesize},
                  ylabel style={font=\footnotesize, inner sep=-5pt},
                  x tick label style={font=\footnotesize},
                  y tick label style={font=\footnotesize},
                  legend style={font=\scriptsize},
                  legend columns = 1,
                  legend style={at={(.75,.95)},anchor=north}]
        \addplot+[mark=*,color=blue] table[x=timeFOM,y=errRefW] {data/errvsCPUandROMsize_N500.txt};
        \addplot+[mark=star,color=red,
        every node near coord/.append style={xshift=0.65cm},
            every node near coord/.append style={yshift=-0.2cm},
            nodes near coords, 
            point meta=explicit symbolic,
            every node near coord/.append style={font=\footnotesize}] 
        table[x=timeROM,y=errRefY, meta index=5] {data/errvsCPUandROMsize_N500.txt}; 
        \legend{Zeitlin $\Iso$, Rec($r$) RKMK-2};
        \nextgroupplot[ylabel={Runtime [s]},
                  xlabel={Rank $r$},
                  axis line style = thick,
                  grid=both,
                  minor tick num=0,
                  xmin = 100, xmax = 500,
                  grid style = {gray,opacity=0.2},
                  xlabel style={font=\footnotesize},
                  ylabel style={font=\footnotesize,inner sep=-2pt},
                  x tick label style={font=\footnotesize},
                  y tick label style={font=\footnotesize}]
        \addplot+[mark=*,color=blue] table[x=ROMsize,y=timeFOM] {data/errvsCPUandROMsize_N500.txt};
         \addplot+[mark=star,color=red] table[x=ROMsize,y=timeROM] {data/errvsCPUandROMsize_N500.txt};
    \end{groupplot}
\end{tikzpicture}
\caption{Left: error between the reference solution and the solution of the Zeitlin model and error between the reference solution and the approximate low-rank solution versus the algorithm runtime. Right: algorithm runtime vs. the approximation rank $r$.}\label{fig:ROMerrvsCPU}
\end{figure}
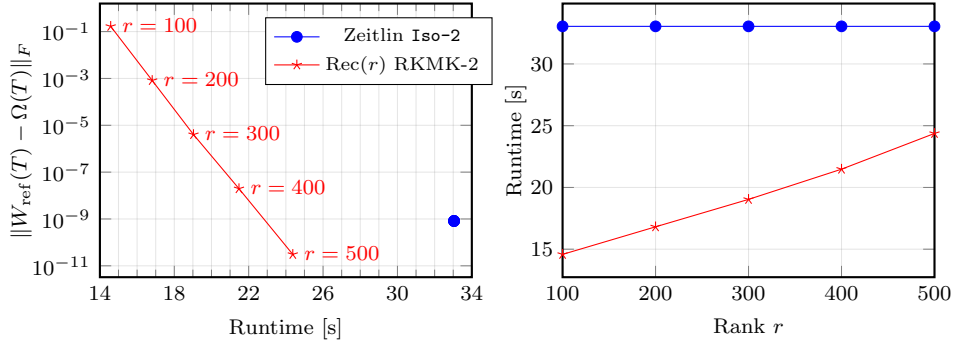

\Cref{fig:ROMerrCasHamvstime} (left) shows that
the error in the Hamiltonian evaluated at the low-rank approximate solution is only due to the quality of the low-rank approximation at the initial time (left plot) and it decreases as $r$ grows, as expected.
The Hamiltonian is preserved by the low-rank approximate trajectories to machine precision and independently of $r$ (not shown here).
To check the numerical conservation of the Casimir invariants we look at the error in the eigenvalues of the approximate solution.
In \Cref{fig:ROMerrCasHamvstime} (right) we report the $\ell^{\infty}$ error between the eigenvalues of the solution at final time and of the initial reference solution. Note that the error is constant in time since both temporal integrators are isospectral and it only depends on the neglected eigenvalues associated with the low-rank factorization.
As expected, the solution obtained with the $\Iso$ solver preserves the eigenvalues and, hence, the Casimirs to machine precision (blue lines with dots). Concerning the low-rank approximation, the error in the eigenvalues decreases as the rank $r$ increases.

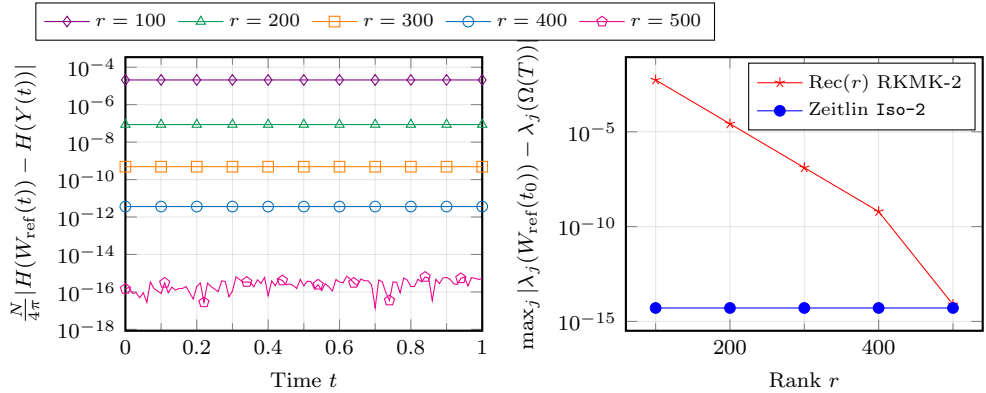
\begin{figure}[H]
\centering
\begin{tikzpicture}
    \begin{groupplot}[
      group style={group size=2 by 1,
                  horizontal sep=1.9cm},
      width=6.3cm, height=5.2cm
    ]
    \nextgroupplot[ylabel={$\frac{N}{4\pi}|H(\Wref(t))-H(Y(t))|$},
                  xlabel={Time $t$},
                  axis line style = thick,
                  grid=both,
                  minor tick num=1,
                  grid style = {gray,opacity=0.2},
                  xmin = 0, xmax = 1,
                  ymode=log,
                  ytick={1e-18,1e-16,1e-14,1e-12,1e-10,1e-8,1e-6,1e-4},
                  xlabel style={font=\footnotesize},
                  ylabel style={font=\footnotesize},
                  x tick label style={font=\footnotesize},
                  y tick label style={font=\footnotesize},
                  legend style={font=\scriptsize},
                  legend columns = 5,
                  legend style={at={(0.7,1.2)},anchor=north}]
        \addplot+[mark=diamond,color=violet,mark repeat=10] table[x=time,y=errHRefY100] {data/errHamvsROMsizeandTime_N500.txt};
        \addplot+[mark=triangle,color=ForestGreen,mark repeat=10] table[x=time,y=errHRefY200] {data/errHamvsROMsizeandTime_N500.txt};
        \addplot+[mark=square,color=orange,mark repeat=10] table[x=time,y=errHRefY300] {data/errHamvsROMsizeandTime_N500.txt};
        \addplot+[mark=o,color=NavyBlue,mark repeat=10] table[x=time,y=errHRefY400] {data/errHamvsROMsizeandTime_N500.txt};
        \addplot+[mark=pentagon,color=magenta,mark repeat=10] table[x=time,y=errHRefY500] {data/errHamvsROMsizeandTime_N500.txt};
        \legend{$r=100\;$,$r=200\;$,$r=300\;$,$r=400\;$,$r=500$};
        \nextgroupplot[ylabel={$\max_j |\lambda_j(\Wref(t_0))-\lambda_j(\Omega(T))|$},
                  xlabel={Rank $r$},
                  axis line style = thick,
                  grid=both,
                  minor tick num=0,
                  grid style = {gray,opacity=0.2},
                  minor tick num=1,
                  ymode = log,
                  xlabel style={font=\footnotesize},
                  ylabel style={font=\footnotesize},
                  x tick label style={font=\footnotesize},
                  y tick label style={font=\footnotesize},
                  legend style={font=\scriptsize},
                  legend cell align=left]
        \addplot+[mark=star,color=red] table[x=ROMsize,y=errEIGWrefY] {data/errEIGYWrefvsROMsize_N500.txt};
        \addplot+[mark=*,color=blue] table[x=ROMsize,y=errEIGWrefW] {data/errEIGYWrefvsROMsize_N500.txt};
        \legend{Rec($r$) RKMK-2, Zeitlin $\Iso$};
    \end{groupplot}
\end{tikzpicture}
\caption{Left: evolution of the error between the Hamiltonian evaluated at the reference solution and the Hamiltonian evaluated at the low-rank approximate solution. Different ranks $r$ are considered.
Right: Error between the eigenvalues of the rank-$r$ approximate solution at final time and the eigenvalues of the reference solution at initial time vs. the rank $r$ of the approximation.
}\label{fig:ROMerrCasHamvstime}
\end{figure}

\subsection{Vortex blob dynamics}

Point-vortex dynamics describes the evolution of solutions where the vorticity is characterized by a finite sum of Dirac distributions.
Point vortices are used to study geophysical turbulence.

In this test we consider an approximation of four point vortices via blobs distributed, at the initial time, according to
$R_j = {\mathrm{exp}}\big(100\sqrt{N}(\ic a_j (T^N_{1,1}-T^N_{1,-1})- b_j (T^N_{1,1}+T^N_{1,-1})+ \ic c_j T^N_{1,0})\big)$, for $j\in\{1,2,3,4\}$,
where $a_j,b_j,c_j$ are random real numbers and the matrices $T^N_{\ell,m}$ are defined according to \eqref{eq:lapl}.
The initial condition is then
$W_0=2\ic \sum_{j=1}^{4} R_j^{\top}B R_j$
where $B\in\mathbb{C}^{N\times N}$ has zero entries except for $B_{N,N}=1$.



We set $N=500$.
The reference solution is obtained by solving the Zeitlin model \eqref{eq:meqt}
with the isospectral Lie-Poisson solver $\Iso$ of \cite{MV20} with $\dt=0.1$.
The reference solution at the initial and final times is reported in \Cref{fig:Wref}. The numerical final time is $T=63000$.

\begin{figure}[H]
\includegraphics[scale=0.57]{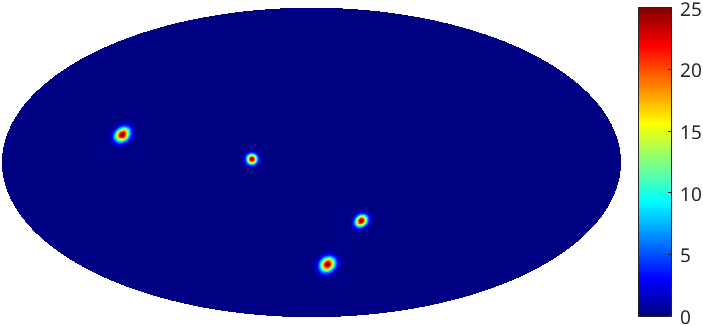}\hspace{-2.2em}
\includegraphics[scale=0.57]{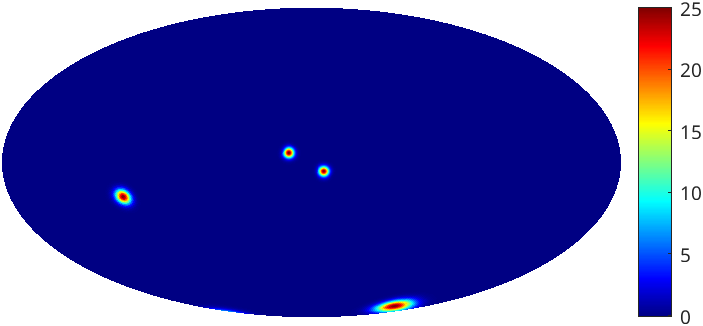}
\caption{Reference solution at the initial time (left) and at the final time (right).}
\label{fig:Wref}
\end{figure}

In this test we compare the performances of the Zeitlin model \eqref{eq:meqt} with initial condition $W_0$ solved with $\Iso$ and the low-rank approximation with rank $r=4$ which corresponds to the actual rank of the vorticity matrix for this test case.
In tests that, for the sake of brevity, we do not report here we observed that there is no improvement in accuracy in performing a low-rank approximation with rank $r>4$ but the computational cost is higher, as expected. 
Taking $r<4$, we observe, at every time, an error proportional to the best rank-$r$ approximation error at the initial time and, thus, proportional to the magnitude of the neglected singular values of $W_0$.
In view of these results we focus on the case $r=4$.

In \Cref{fig:blob_errvscpu} we compare the performances of the Zeitlin model
and of the rank-$4$ approximation obtained from \eqref{eq:US0} in terms of accuracy and computational cost.
We let the time step $\dt$ vary and compute the errors with respect to the reference solution at the final time $T$. We observe that, for a fixed $\dt$, the low-rank approximation yields a smaller error at a lower computational cost.

\begin{figure}[H]
\centering
\begin{tikzpicture}
    \begin{axis}[width = 8.5cm, height=5.5cm,
                ylabel={$\normF{\Wref(T)-\Omega(T)}$},
                  xlabel={Runtime [s]},
                  axis line style = thick,
                  grid=both,
                  minor tick num=3,
                  grid style = {gray,opacity=0.2},
                  xmax = 1e+5,
                  ymin = 1e-3, ymax = 10,
                  ymode=log,
                  xmode=log,
                  xlabel style={font=\footnotesize},
                  ylabel style={font=\footnotesize, inner sep=0pt},
                  x tick label style={font=\footnotesize},
                  y tick label style={font=\footnotesize},
                  legend style={font=\scriptsize},
                  legend columns = 1,
                  legend style={at={(0.25,0.25)},anchor=north},
                  legend cell align=left]
        \addplot+[mark=star,color=red,
        every node near coord/.append style={xshift=0.65cm},
            every node near coord/.append style={yshift=-0.01cm},
            nodes near coords, 
            point meta=explicit symbolic,
            every node near coord/.append style={font=\footnotesize}] table[x=cpuYMKN,y=errWrefYMKN,meta index=10] {data/blobs_errvsCPU_N500r4.txt};
        \addplot+[mark=*,color=blue,
        every node near coord/.append style={xshift=0.65cm},
            every node near coord/.append style={yshift=-0.01cm},
            nodes near coords, 
            point meta=explicit symbolic,
            every node near coord/.append style={font=\footnotesize}] 
        table[x=cpuW,y=errWrefW, meta index=10] {data/blobs_errvsCPU_N500r4.txt}; 
        \legend{Rec($4$) RKMK-2,Zeitlin $\Iso \quad$};
    \end{axis}
\end{tikzpicture}
\caption{Error, at the final time, between the reference solution and the rank-$4$ solution \eqref{eq:fact} obtained from \eqref{eq:US0} and error between the reference solution and the solution of the Zeitlin model \eqref{eq:meqt} versus the algorithm runtime. Different values of the time step $\Delta t$ are considered.}
\label{fig:blob_errvscpu}
\end{figure}
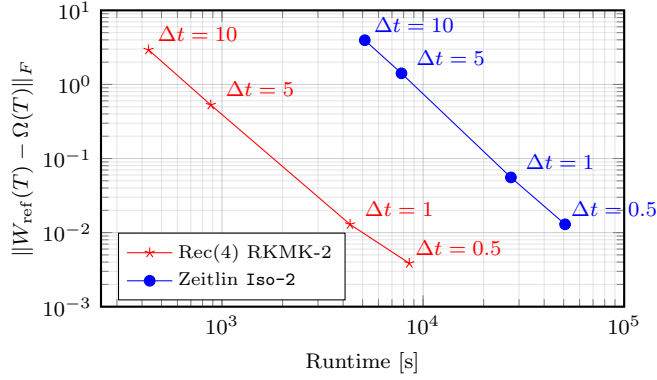

In a second set of numerical experiments we fix the time step $\dt=1$ and study the performances of the rank-$4$ approximation with and without truncation of the stream matrix; see \Cref{sec:appot}.

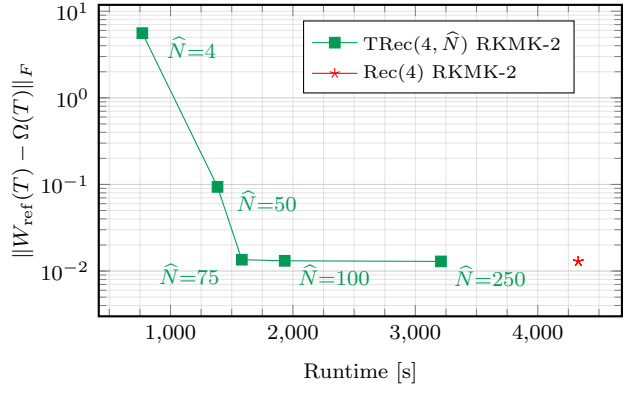
\begin{figure}[H]
\centering
\begin{tikzpicture}
    \begin{axis}[width = 8.5cm, height=5.7cm,
                ylabel={$\normF{\Wref(T)-\Omega(T)}$},
                  xlabel={Runtime [s]},
                  axis line style = thick,
                  grid=both,
                  minor tick num=3,
                  grid style = {gray,opacity=0.2},
                  ymin = 3e-3,
                  ymode=log,
                  xlabel style={font=\footnotesize},
                  ylabel style={font=\footnotesize, inner sep=-1pt},
                  x tick label style={font=\footnotesize},
                  y tick label style={font=\footnotesize},
                  legend style={font=\scriptsize},
                  legend columns = 1,
                  legend style={at={(0.65,0.95)},anchor=north},
                  legend cell align=left]
        \addplot+[mark=square*,color=ForestGreen,
             every node near coord/.append style={xshift=0.65cm},
             every node near coord/.append style={yshift=-0.45cm},
             anchor = east,
             nodes near coords, 
             point meta=explicit symbolic,
             every node near coord/.append style={font=\footnotesize}]      table[x=cpuZMKN,y=errWrefZMKN, meta index=11] {data/blobs_hrom_errvsCPU_N500r4.txt}; 
        \addplot+[mark=star,color=red] table[x=cpuYMKN,y=errWrefYMKN] {data/blobs_hrom_errvsCPU_N500r4.txt};
        \legend{TRec({$4,\ap{N}$}) RKMK-2, Rec($4$) RKMK-2};
    \end{axis}
\end{tikzpicture}
\caption{Error of the rank-$r$ approximate solution from \eqref{eq:US0} and error of the rank-$r$ approximate solution from the reconstruction equation with truncation \eqref{eq:apUS0} versus the algorithm runtime, for different values of the truncation size $\ap{N}$.}
\label{fig:blob_hrom_errvscpu}
\end{figure}

\Cref{fig:blob_hrom_errvscpu} reports the error at the final time vs. the algorithm runtime of the rank-$4$ approximation with and without the truncation.
It can be observed that the rank-$4$ reconstruction equation with truncation is computationally cheaper than solving the rank-$4$ reconstruction equation without truncation but at the cost of a decrease in accuracy when $\ap{N}<75$. For a truncation size $\ap{N}$ equal to $75$ or larger the rank-$4$ approximation with truncation achieves the error of the rank-$4$ approximation, and, for $\ap{N}=75$, at less than half of the computational cost.
The Zeitlin model requires $27334$s 
to achieve an error of $5.55\rm{e}{-}2$, and it is not reported in the figure. This means that it is roughly 17 times more expensive than the truncated low-rank pproximation with $\ap{N}=75$ and has even a slightly larger error.

In \Cref{fig:blobs_errHam}, on the left, we report the error in the conservation of the Hamiltonian. 
Exact conservation of the Hamiltonian is not expected from the proposed numerical time integration schemes. We observe that the reference solution has a smaller error than the other trajectories owing to a smaller time step.
The low-rank approximations have a similar behavior to the full-rank solution of the Zeitlin model with a slightly lower error.
On the right plot of \Cref{fig:blobs_errHam}, the evolution of the error between the Hamiltonian evaluated at the reference solution and at the low-rank approximation with truncation is reported. Different truncation sizes $\ap{N}$ are considered. As predicted by the theory, the truncated model is Hamiltonian with a Hamiltonian $\ap{H}$ \eqref{eq:apHam} that is an approximation of the original one. As $\ap{N}$ increases, the error in the approximation of the Hamiltonian decreases, and for $\ap{N}>75$ the truncation is no longer affecting the accuracy of the approximation.

\begin{figure}[H]
\centering
\begin{tikzpicture}
    \begin{groupplot}[
      group style={group size=2 by 1,
                  horizontal sep=1.7cm},
      width=6.4cm, height=5.7cm
    ]
    \nextgroupplot[ylabel={$\frac{N}{4\pi}|H(\Wref(t_0))-H(\Omega(t))|$},
                  xlabel={Time $t$},
                  axis line style = thick,
                  grid=both,
                  minor tick num=1,
                  grid style = {gray,opacity=0.2},
                  ymode=log,
                  ymax = 1e-5, ymin=1e-20,
                  ytick={1e-20,1e-17,1e-14,1e-11,1e-8,1e-5},
                  xlabel style={font=\footnotesize},
                  ylabel style={font=\footnotesize},
                  x tick label style={font=\footnotesize},
                  y tick label style={font=\footnotesize},
                  legend style={font=\scriptsize},
                  legend columns = 1,
                  legend style={at={(0.575,0.5)},anchor=north},
                  legend cell align=left]
        \addplot+[mark=*,color=blue,mark repeat=11] table[x=time,y=consHWdt1] {data/blobs_consHamvstime_N500r4.txt};
        \addplot+[mark=star,color=red,mark repeat=11] table[x=time,y=consHYMKNdt1] {data/blobs_consHamvstime_N500r4.txt};
        \addplot+[mark=square,color=ForestGreen,mark repeat=11] table[x=time,y=consHZMKNdt1m75] {data/blobs_consHamvstime_N500r4.txt};
        \addplot+[mark=diamond*,color=black,mark repeat=11] table[x=time,y=consHWref] {data/blobs_consHamvstime_N500r4.txt};
        \legend{Zeitlin $\Iso$, Rec($4$) RKMK-2,TRec{$(4,75)$} RKMK-2, $\Wref$};
        \nextgroupplot[ylabel={$\frac{N}{4\pi}|H(\Wref(t))-H(Z(t))|$},
                  xlabel={Time $t$},
                  axis line style = thick,
                  grid=both,
                  minor tick num=0,
                  grid style = {gray,opacity=0.2},
                  minor tick num=1,
                  ymode = log,
                  ymin=1e-16,
                  ytick={1e-16,1e-13,1e-10,1e-7,1e-4,1e-1},
                  xlabel style={font=\footnotesize},
                  ylabel style={font=\footnotesize},
                  x tick label style={font=\footnotesize},
                  y tick label style={font=\footnotesize},
                  legend style={font=\scriptsize},
                  legend columns = 1,
                  legend style={at={(0.8,0.55)},anchor=north},
                  legend cell align=left]
        \addplot+[mark=diamond,color=violet,mark repeat=10] table[x=time,y=errHZMKNdt1m4] {data/blobs_hrom_errHamvstime_N500r4.txt};
        \addplot+[mark=triangle,color=orange,mark repeat=10] table[x=time,y=errHZMKNdt1m50] {data/blobs_hrom_errHamvstime_N500r4.txt};
        \addplot+[mark=square,color=ForestGreen,mark repeat=10] table[x=time,y=errHZMKNdt1m75] {data/blobs_hrom_errHamvstime_N500r4.txt};
        \addplot+[mark=o,color=NavyBlue,mark repeat=10] table[x=time,y=errHZMKNdt1m100] {data/blobs_hrom_errHamvstime_N500r4.txt};
        \addplot+[mark=pentagon,color=magenta,mark repeat=10] table[x=time,y=errHZMKNdt1m250] {data/blobs_hrom_errHamvstime_N500r4.txt};
         \legend{$\ap{N}=4$,$\ap{N}=50$,$\ap{N}=75$,$\ap{N}=100$,$\ap{N}=250$};
    \end{groupplot}
\end{tikzpicture}
\caption{
Left: Conservation of the Hamiltonian.
Right: evolution of the error between the Hamiltonian evaluated at the reference solution and the Hamiltonian evaluated at the low-rank approximate solution of \eqref{eq:meqt_hypred} for different values of the truncation size $\ap{N}$.
}\label{fig:blobs_errHam}
\end{figure}
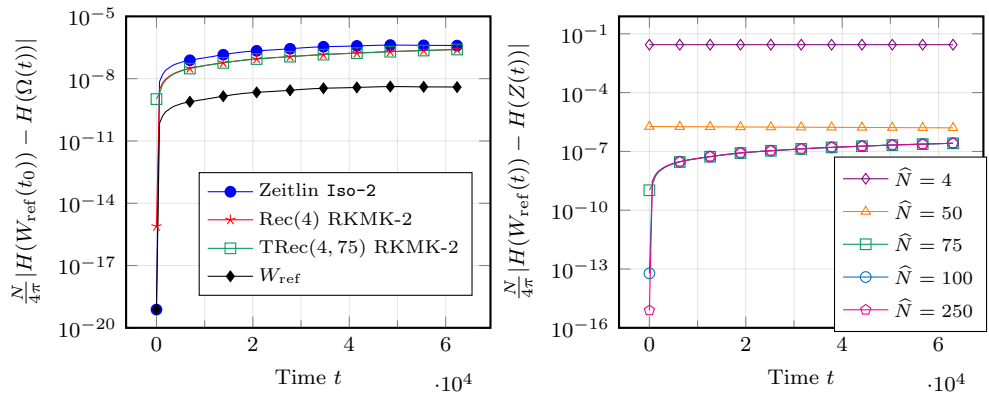

The conservation of the Casimir functions is numerically studied in \Cref{fig:blobsEIG}, where we plot the evolution of the $\ell^{\infty}$ error between the eigenvalues of the solution at each time and of the initial reference solution.
While some numerical errors are affecting the solution of the Zeitlin model, the low-rank approximations preserve the $4$ dominant eigenvalues to almost machine precision.

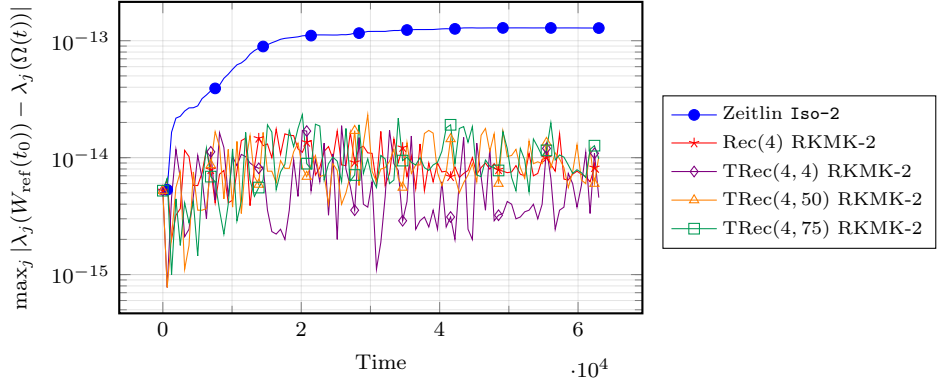
\begin{figure}[H]
\centering
\begin{tikzpicture}
    \begin{axis}[width=8.5cm, height=5.7cm,
                  ylabel={$\max_j |\lambda_j(\Wref(t_0))-\lambda_j(\Omega(t))|$},
                  xlabel={Time},
                  axis line style = thick,
                  grid=both,
                  minor tick num=0,
                  grid style = {gray,opacity=0.2},
                  minor tick num=1,
                  ymode = log,
                  xlabel style={font=\footnotesize},
                  ylabel style={font=\footnotesize},
                  x tick label style={font=\footnotesize},
                  y tick label style={font=\footnotesize},
                  legend style={font=\scriptsize},
                  legend style={at={(1.3,0.7)},anchor=north},
                  legend cell align=left]
        \addplot+[mark=*,color=blue,mark repeat=11] table[x=time,y=eW2] {data/blobs_errEIGvstime_N500r4.txt};
        \addplot+[mark=star,color=red,mark repeat=11] table[x=time,y=eYMKN2] {data/blobs_errEIGvstime_N500r4.txt};
        \addplot+[mark=diamond,color=violet,mark repeat=11] table[x=time,y=eZMKNA] {data/blobs_errEIGvstime_N500r4.txt};
        \addplot+[mark=triangle,color=orange,mark repeat=11] table[x=time,y=eZMKNB] {data/blobs_errEIGvstime_N500r4.txt};
        \addplot+[mark=square,color=ForestGreen,mark repeat=11] table[x=time,y=eZMKNC] {data/blobs_errEIGvstime_N500r4.txt};
        \legend{Zeitlin $\Iso$, Rec($4$) RKMK-2, TRec{$(4,4)$} RKMK-2,TRec{$(4,50)$} RKMK-2,TRec{$(4,75)$} RKMK-2,TRec{$(4,100)$} RKMK-2,TRec{$(4,250)$ RKMK-2}};
    \end{axis}
\end{tikzpicture}
\caption{Error between the eigenvalues of the low-rank approximate solution at final time and the eigenvalues of the reference solution at initial time vs. the rank $r$ of the approximation.}\label{fig:blobsEIG}
\end{figure}

As a final test, we consider the low-rank approximation with the factorization proposed in \Cref{sec:splitting}, namely where the factor $S$ is time-dependent. The problem parameters are as described above.
In the temporal splitting of \Cref{sec:splitting}, the evolution equation for $U$ in \eqref{eq:U} is solved using a RK-MK integrator of order $2$, while the evolution of $S$ in \eqref{eq:S} is solved with $\Iso$.
In \Cref{table:St} we compare the performances of the Zeitlin model, low-rank approximation \eqref{eq:fact} with fixed factor $S_0$, and the low-rank approximation \eqref{eq:factSt} with time-dependent $S$.
We observe that the low-rank approximation with splitting is slightly more accurate than the other models, both in the solution and in the Hamiltonian conservation. Compared to the factorization with $S_0$ fixed, the model with time-dependent $S$ is computationally considerably more expensive (by almost 4.5 times), although it is cheaper than solving the Zeitlin model \eqref{eq:meqt}.
This suggests that the factorization \eqref{eq:factSt} together with the temporal splitting of \Cref{sec:splitting} is a valid low-rank approximation and it is best suited for flows on matrix manifolds that are not isospectral.

\renewcommand{\arraystretch}{1.75}
\begin{table}[H]
\footnotesize
\caption{Comparison of the Zeitlin model \eqref{eq:meqt}, the low-rank approximation with factorization \eqref{eq:fact} and the low-rank approximation \eqref{eq:factSt} with time-dependent $S$.}
\label{table:St}
\begin{center}
    \begin{tabular}{c|c|c|c}
         & Zeitlin with $\Iso$  & Rec($4$) with $S_0$ & Rec($4$) with $S(t)$ \\ \hline
        $\normF{\Wref(T)-\Omega(T)}$ & 5.55\rm{e}{-}2 & 1.29\rm{e}{-}2 & 5.30\rm{e}{-}3 \\ \hline
        $\frac{N}{4\pi}\max_{\tind}|H(\Wref(t_0))-H(\Omega(t_{\tind}))|$ & 4.15\rm{e}{-}7 & 2.57\rm{e}{-}7 & 3.21\rm{e}{-}8 \\ \hline
        Runtime [s] & 27334 & 4331 & 19322
    \end{tabular}
\end{center}
\end{table}

\section{Concluding remarks}
\label{sec:concl}
We have proposed a low-rank approximation of the Zeitlin model that provides a finite-dimensional description of the incompressible Euler equations on the sphere.
Two factorizations of the vorticity matrix have been considered: one based on a eigendecomposition where only the basis of eigenvectors depends on time, and a second one where all factors are time-dependent.
Despite having both structure-preserving and favorable approximability properties, the first turns out to be more suitable for isospectral flow, while the second factorization might be used for other type of Euler--Arnold equations. Extensive numerical experiments on the latter are left for future work.

\section*{Acknowledgments}
The author would like to thank Klas Modin for inspiring discussions on various aspects of the Zeitlin model.
Several discussions with Arnout Franken and Erwin Luesink in a preliminary stage of this work, and with Milo Viviani are also gratefully acknowledged.

\printbibliography

\end{document}